\documentclass[11pt]{amsart}
\usepackage{amsmath, amssymb, amsthm, amsfonts, mathrsfs, eucal, enumerate, 
layout}
\usepackage[normalem]{ulem}
\usepackage[all,tips,2cell]{xy}
\usepackage[usenames]{color}
\usepackage[T1]{fontenc}
\usepackage{xspace}
\usepackage{ifpdf}
\usepackage{ifthen}
\usepackage{verbatim}
\usepackage[utf8]{inputenc}

\newtheorem{theorem}{Theorem}[section]
\newtheorem*{theorem*}{Theorem}
\newtheorem{corollary}[theorem]{Corollary}
\newtheorem*{corollary*}{Corollary}
\newtheorem{lemma}[theorem]{Lemma}
\newtheorem*{lemma*}{Lemma}

\newtheorem*{claim*}{Claim}
\newtheorem{proposition}[theorem]{Proposition}
\newtheorem*{proposition*}{Proposition}
\newtheorem{def-prop}[theorem]{Definition-Proposition}
\theoremstyle{definition}
\newtheorem{definition}[theorem]{Definition}
\newtheorem*{definition*}{Definition}
\newtheorem{remark}[theorem]{Remark}

\newtheorem{example}[theorem]{Example}
\newtheorem*{example*}{Example}
\numberwithin{equation}{section}

\newcommand{\bS}{\mathbf{S}}

\newcommand{\ZZ}{\mathbb{Z}}
\newcommand{\QQ}{\mathbb{Q}}

\newcommand{\GG}{\mathbb{G}}

\newcommand{\CUB}{\mathrm{CUB}}

\newcommand{\champcub}{{\mathcal Cub}}

\newcommand{\Hom}{\mathrm{Hom}}
\newcommand{\Ext}{\mathrm{Ext}}
\newcommand{\Aut}{\mathrm{Aut}}
\newcommand{\TR}{\mathrm{RLB}}
\newcommand{\uHom}{\underline{\mathrm{Hom}}}

\newcommand{\bPicard}{\mathrm{\mathbf{Picard}}}

\newcommand{\W}{\mathrm{W}}

\newcommand{\HH}{\mathrm{H}}
\newcommand{\R}{\mathrm{R}}

\newcommand{\id}{\mathrm{id}}
\newcommand{\st}{\mathrm{st}}

\newcommand{\cD}{\mathcal{D}}
\newcommand{\cG}{\mathcal{G}}

\newcommand{\cL}{\mathcal L}
\newcommand{\cM}{\mathcal{M}}
\newcommand{\cN}{\mathcal{N}}

\newcommand{\cQ}{\mathcal{Q}}
\newcommand{\cO}{\mathcal{O}}
\newcommand{\cX}{\mathcal{X}}
\newcommand{\cY}{\mathcal{Y}}

\newcommand{\cHom}{{\mathcal Hom}}

\newcommand{\PIC}{\mathrm{PIC}}
\newcommand{\Pic}{\mathrm{Pic}}

\newcommand{\ov}{\overline}

\def\Gm{\mathbb{G}_\textrm{m}}
\newcommand{\Sgo}{\mathfrak{S}}
\def\tore{T}
\def\vt{\widetilde{v}}
\def\philt{\widetilde{\varphi_L}}
\newcommand{\cartesien}{\ar@{}[dr]|{\square}}

\begin{document}

\UseAllTwocells
\title[line bundles]
{Morphisms of 1-motives defined by line bundles}

\author{Cristiana Bertolin}
\address{Dip. di Matematica, Universit\`a di Torino, Via Carlo Alberto 10, 
I-10123 Torino}
\email{cristiana.bertolin@googlemail.com}

\author{Sylvain Brochard}
\address{IMAG, Université de Montpellier, CNRS, Montpellier, France}
\email{sylvain.brochard@univ-montp2.fr}

\subjclass{14K30,14C20}

\keywords{1-motives, line bundles, linear morphisms, commutative 
group stack}

\date{}
\dedicatory{}

\commby{Cristiana Bertolin}


\begin{abstract}
Let $S$ be a normal base scheme. The aim of this paper is to study the line bundles on 1-motives defined over~$S$. 
We first compute a \emph{d\'evissage} of the Picard group of a 1-motive $M$
according to the weight filtration of $M$. This d\'evissage allows us 
to associate, to each line bundle $\cL$ on $M$, a linear morphism $\varphi_{\cL}: M \rightarrow M^*$ 
from $M$ to its Cartier dual. This yields a group homomorphism $\Phi : 
\Pic(M) / \Pic(S) \to \Hom(M,M^*)$.
We also prove the \emph{Theorem of the Cube} for 1-motives, which furnishes another construction of the group homomorphism  $\Phi : 
\Pic(M)  / \Pic(S) \to \Hom(M,M^*)$. 
Finally we prove that these two independent constructions of linear morphisms $M \to M^*$ using 
line bundles on $M$ coincide. However, the first construction, involving the d\'evissage of $\Pic(M)$, is more explicit and geometric and it furnishes the 
\emph{motivic origin} of some linear morphisms between 1-motives. 
The second construction, involving the Theorem of the Cube, is more abstract but also more enlightening.
\end{abstract}


\maketitle


\setcounter{tocdepth}{1}
\tableofcontents

\section{Introduction}

Let $A$ be an abelian variety over a field $k$ and let $A^*=\Pic^0_{A/k}$ be 
its dual. It is a classical fact that if $L$ is a line bundle on $A$, then the 
morphism $\varphi_L : A\to A^*$, defined by $\varphi_L(a)=\mu_a^*L\otimes 
L^{-1}$ where $\mu_a : A\to A$ is the translation by $a$, is a group 
homomorphism. This is an easy consequence of the Theorem of the Square, which 
itself is a consequence of the Theorem of the Cube.
We then have a functorial homomorphism  
$\Phi :\Pic(A)   \to \Hom(A,A^*)$ which is a key result in the basic foundations 
of the theory of abelian varieties.
In~\cite[Section 10]{D1} Deligne introduced the notion of 
1-motives, which can be seen as a generalization of abelian schemes. Let $S$ be 
a scheme. A 1-motive $M=(X,A,\tore,G,u)$ defined over $S$ is a complex
$[u: X \rightarrow G]$ of commutative $S$-group schemes concentrated 
in degree 0 and -1, where:
\begin{itemize}
    \item $X$ is an $S$-group scheme which is locally for the \'etale
topology a constant group scheme defined by a finitely generated free
$\ZZ \,$-module,
    \item $G$ is an extension of an abelian $S$-scheme $A$ by an $S$-torus 
$\tore,$
    \item $u:X \rightarrow G$ is a morphism of $S$-group schemes.
\end{itemize}
A linear morphism of 1-motives is a morphism of complexes of $S$-group schemes. 
We will denote by 
\[\Hom(M_1,M_2)\]
the group of linear morphisms from $M_1$ to $M_2$.
In this paper we study line bundles on a 1-motive $M$ and their 
relation to linear morphisms from $M$ to its Cartier dual $M^*$.

Our aim is to answer the following natural questions:
\begin{enumerate}
\item If $M$ is a 1-motive over $S$, is it 
possible to construct a functorial homomorphism  
$\Phi :\Pic(M)   \to \Hom(M,M^*)$ that extends the known one for 
abelian schemes?
\item Is there an analog of the Theorem of the Cube for 1-motives?
\end{enumerate}

We give a positive answer to both questions if the base scheme $S$ is normal 
(for comments on what happens if the base scheme $S$ is not 
normal, see Remark~\ref{rem:comments_normalness}).
 
The notion of line bundle on a 1-motive $M$ over $S$ already 
implicitly exists in the literature. Actually, in~\cite[p. 64]{Mumford65} 
Mumford introduced a natural notion of line bundles 
on an arbitrary $S$-stack $\cX$ (see~\ref{def-linebundle-stack}).
Since to any 1-motive $M$ over $S$ 
we can associate by~\cite[ \S 1.4]{SGA4} a commutative group stack 
$\st(M)$, we can define the category $\PIC(M)$ of line bundles on $M$ 
as the category of line bundles on $\st(M)$. The \emph{Picard group of} $M$, 
denoted by $\Pic(M)$, is the group of isomorphism classes of line bundles on $\st(M)$ 
(see Definition~\ref{def-linebundle-1-motive}).

The stack $\st(M)$ associated to a 1-motive $M=[X 
\stackrel{u}{\rightarrow}G]$ is isomorphic to the quotient 
stack $[G/X]$, where $X$ acts on $G$ by translations via $u$. Under this 
identification, the inclusion of 1-motives $\iota : G\to M$ corresponds to the 
projection map $G \to [G/X]$, which is étale and surjective. We can then 
describe line bundles on $M$ as couples 
\[(L, \delta)\]
where $L$ is a line bundle on $G$ and $\delta$ is a 
descent datum for $L$ with respect to the covering $\iota : G\to [G/X]$ (see Section~\ref{section:line_bundles}, after Lemma~\ref{lem:equiv-def}). Throughout this paper, we will use 
this description of line bundles on $M$, which amounts to say that 
a line bundle on a 1-motive $M$ is a line bundle on $G$ endowed with an action of $X$ that 
is compatible with the translation action of $X$ on $G$.

The main result of our paper is the following theorem, which generalizes to 
1-motives the classical homomorphism $\Phi :\Pic(A)   \to \Hom(A,A^*)$ for
abelian varieties.

\begin{theorem}\label{thm:existence_Phi}
Let $M$ be a 1-motive defined over a scheme $S$. Assume that the 
toric part of $M$ is trivial or that $S$ is normal. 
Then there is a functorial homomorphism
\begin{equation}
\Phi :  \Pic(M)/\Pic(S) \longrightarrow \Hom(M,M^*).
\label{eq:PHI}
\end{equation}
\end{theorem}

We actually provide two independent constructions of~$\Phi$:
\begin{enumerate}
	\item The first construction, given in 
Section~\ref{section:direct_construction}, is the most explicit and 
geometric one. It is based on the ``\emph{d\'evissage}'' of the Picard group of $M$, computed in 
Section~\ref{section:devissage}, and on the explicit functorial description of the 
Cartier dual~$M^*$ of $M$ in terms of extensions given in~\cite[(10.2.11)]{D1}. 
\item The second construction, given in 
Sections~\ref{section:linear_morphisms_def_by_cubical} 
and~\ref{section:thm_cube}, is more abstract but also more enlightening. It 
works for a category which is a bit larger than 1-motives (see~\ref{thm:cube}) 
and it also provides the fact that $\Phi$ is a group homomorphism. This 
construction relies on the ``\emph{Theorem of the Cube for 1-motives}'' (Theorem \ref{thm:cube}), a result that 
we think is of independent interest, and on the description of the Cartier dual of a 
1-motive in terms of commutative group stacks.
\end{enumerate}

In Proposition~\ref{thm:comparaison} we prove that these two constructions coincide.

\textbf{D\'evissage of the Picard group of $M$}: 1-motives are endowed with a weight filtration $\W_*$ defined by
$ {\W}_{0}(M) = M,  {\W}_{-1}(M) = G, {\W}_{-2}(M) = \tore, {\W}_{j}(M) = 0 $ for each $ j \leq -3.$
This weight filtration allows us to ``\emph{d\'evisser}'' the Picard group of $M$, which is our second main result:
 we will first describe the Picard group of $G$ in terms of $\Pic(A)$ and $\Pic(\tore)$ using the first 
short exact sequence $0 \rightarrow  \tore \stackrel{i}{\rightarrow} G
\stackrel{\pi}{\rightarrow}  A \rightarrow  0$ given by $\W_*$.
Consider the morphism 
\[
\xi : \Hom(\tore,\GG_m) \rightarrow  \Pic (A)
\label{eq:xi}
\]
defined as follows: for any morphism of $S$-group schemes  $\alpha:\tore \to 
\GG_m $, 
$\xi(\alpha)$ is the image of the class $[\alpha_* G ]$ of the push-down of $G$ via $\alpha$
under the inclusion $ \Ext^1(A,\GG_m) \hookrightarrow H^1(A,\GG_m) =\Pic (A)$.
At the beginning of Section \ref{section:devissage} we will show that

\begin{proposition}\label{prop:TGA}
Assume the base scheme $S$ to be normal. The following 
sequence of groups is exact
\[
0\longrightarrow
\Hom(G,\GG_m) \stackrel{i^*}{\longrightarrow}  
\Hom(\tore,\GG_m) \stackrel{\xi}{\longrightarrow} 
\frac{\Pic(A)}{\Pic(S)} \stackrel{\pi^*}{\longrightarrow}  
\frac{\Pic (G)}{\Pic(S)} \stackrel{i^*}{\longrightarrow} 
\frac{\Pic(\tore)}{\Pic(S)}.
\]
\end{proposition}

The second short exact sequence
$0 \rightarrow  G \stackrel{\iota}{\rightarrow} M
\stackrel{\beta}{\rightarrow}  X[1] \rightarrow  0$ given by the weight 
filtration $\W_*$ of $M$ induces by pullback the sequence
$\Pic(X[1]) \stackrel{\beta^*}{\rightarrow}
\Pic(M)  \stackrel{\iota^*}{\rightarrow}
\Pic(G), $
which is not exact as we will see in
Example~\ref{sequence_not_exact}, but which is nevertheless interesting since the 
kernel of the homomorphism $\iota^* : \Pic(M)\to \Pic(G)$ fits in a long
exact sequence. In fact, at the end of Section \ref{section:devissage} we will prove that

\begin{proposition}
\label{prop:devissage_kernel}
Assume the base scheme $S$ to be reduced. Then 
the kernel $K$ of the homomorphism $\iota^* : \Pic(M)\to 
\Pic(G)$ fits in an
exact sequence
\[
 \Hom(G,\GG_m) \stackrel{\circ u}{\longrightarrow}  \Hom(X,\GG_m) \stackrel{\beta^*}{\longrightarrow} K \stackrel{\Theta}{\longrightarrow} \Lambda
 \stackrel{\Psi}{\longrightarrow} \Sigma.
\]
Note that the group $\Hom(X,\GG_m)$ in the above sequence identifies in a natural way with 
$\Pic(X[1])/\Pic(S)$.
\end{proposition}

Here the group $\Lambda$ is the
subgroup of $\Hom(X,G^D)$, where $G^D$ is the group scheme $\underline{\Hom}(G, 
\GG_m)$,
consisting 
of those morphisms of $S$-group schemes that satisfy the
equivalent conditions of Lemma \ref{lem:equivalentconditions}, and 
 $\Sigma$ is a quotient of the group of 
symmetric bilinear morphisms $X\times_S X \to \GG_m$ 
(see Definition \ref{Lambda} and (\ref{eq:Theta}) for the definitions of  
$\Lambda, \Sigma, \Psi$ and $\Theta$). Remark that there is a natural identification of $K$ with the kernel of $\Pic(M) /\Pic(S) \to \Pic(G) /\Pic(S)$ and so the map $\beta^*$ in the above sequence is really the 
pullback along $\beta : M \to X[1]$.

\textbf{Theorem of the Cube for 1-motives}: In its classical form, the Theorem of the Cube asserts that for any line 
bundle $L$ on an abelian variety, the associated line bundle $\theta(L)$ 
is trivial (see Section~\ref{section:linear_morphisms_def_by_cubical} for the 
definitions of $\theta(L)$ and $\theta_2(L)$). 
In~\cite{Breen_Fonctions_Theta} Breen proposed  the following \emph{reinforcement of the Theorem of 
the Cube}. A cubical structure on $L$ is a section of~$\theta(L)$ that satisfies 
some additional conditions so that $\theta_2(L)$ is endowed with a structure of 
symmetric biextension. A cubical line bundle is a line bundle endowed with a 
cubical structure. Then a commutative $S$-group scheme $G$ is said to satisfy the (strengthened form of 
the) Theorem of the Cube if the forgetful functor 
$$ \CUB(G) \longrightarrow \TR(G)$$ 
from the category $\CUB(G)$ of 
cubical line bundles on $G$ to the category $\TR(G)$ of rigidified line 
bundles on $G$ is an equivalence of categories.

The notion of cubical structure introduced by Breen generalizes seamlessly 
to commutative group stacks (see Definition \ref{def:cub}). In a very general 
context, in Theorem \ref{thm:morphismes_induits_par_fi_cubistes}, we explain how 
a cubical line bundle $(\cL,\tau)$ on a commutative group stack $\cG$ 
defines an additive 
functor from $\cG$ to its dual $ D(\cG)={\cHom}(\cG,B\GG_m)$:
\[
\begin{array}{lrcl}
\varphi_{(\cL,\tau)} : & \cG & \longrightarrow & D(\cG) \\
    & a & \longmapsto & \big(b\mapsto  \cL_{ab}\otimes \cL_a^{-1} \otimes \cL_b^{-1}\big) \, .
\end{array}
\]

In Theorem \ref{thm:cube} we show that over a normal base scheme, 
1-motives satisfy the Theorem of the Cube in the above sense, which is our 
third main result. Then Theorem~\ref{thm:existence_Phi} is an immediate 
corollary of Theorems~\ref{thm:morphismes_induits_par_fi_cubistes} 
and ~\ref{thm:cube}.
Remark that the quotient $\Pic (M) /\Pic(S)$ 
is isomorphic to the group of isomorphism classes of rigidified line 
bundles on $M$.

We finish observing that the construction of the morphism $\Phi(L,\delta): M \to M^*$,
with $(L,\delta)$ a line bundle on $M$, that we give in Section~\ref{section:direct_construction}, 
is completely geometric and so it allows the computation of the Hodge, the De Rham and the $\ell$-adic 
realizations of $\Phi(L,\delta): M \to M^*,$ with their comparison isomorphisms. 
This furnishes the motivic origin of some linear morphisms between 1-motives and their Cartier duals (here \emph{motivic} means coming from geometry - see \cite{D}). In this setting, an ancestor of this paper is 
 \cite{B09} where the first author defines the notion of biextensions of 1-motives and shows that such biextensions furnish bilinear morphisms between 1-motives in the Hodge, the De Rham and the $\ell$-adic 
realizations. Just as biextensions of 1-motives are the motivic origin of bilinear morphisms between 1-motives, line bundles on a 1-motive $M$ are the motivic origin of some linear morphisms between $M$ and its Cartier dual $M^*.$
As observed in Remark \ref{PhiNotSurjective} not all morphisms from $M$ to $M^*$ are defined by line bundles.


\section{Notation}

 Let $\bS$ be a site. For the definitions of $\bS$-stacks and the related 
vocabulary we refer to~\cite{Giraud}. By a stack we 
always mean a stack in groupoids. If $\cX$ and $\cY$ are two $\bS$-stacks, 
$\cHom_{\bS-stacks}(\cX,\cY)$ will be the $\bS$-stack such that 
 for any object $U$ of $\bS$, $\cHom_{\bS-stacks}(\cX,\cY)(U)$ is the category 
of morphisms of $\bS$-stacks from ${\cX}_{|U}$ to ${\cY}_{|U}$. If $S$ is a scheme, an~$S$-stack will be a 
stack for the \emph{fppf} topology.

A \emph{commutative group $\bS$-stack} is an $\bS$-stack $\cG$ endowed with a 
functor $ +: \cG \times_{\bS} \cG \rightarrow 
\cG, ~~(a,b) \mapsto a+b$, and two natural isomorphisms of associativity 
$\sigma$ and of commutativity $\tau$,
such that for any object $U$ of $\bS$, $(\cG (U),+,\sigma, \tau)$ is a strictly 
commutative Picard category.
 An \emph{additive functor} $(F,\sum):\cG_1 \rightarrow \cG_2 $
between two commutative group $\bS$-stacks is a morphism of ${\bS}$-stacks $F: 
\cG_1 
\rightarrow \cG_2$ endowed with a natural isomorphism $\sum: F(a+b) \cong 
F(a)+F(b)$ (for all $a,b \in \cG_1$) which is compatible with the natural 
isomorphisms $\sigma$ and $\tau$ underlying 
$\cG_1$ and $\cG_2$.  A \emph{morphism of additive functors $u:(F,\sum) 
\rightarrow (F',\sum') $} is an ${\bS}$-morphism of cartesian ${\bS}$-functors 
(see \cite[Chp I 1.1]{Giraud}) which is compatible with the natural isomorphisms 
$\sum$ and $\sum'$ of $F$ and $F'$ respectively. For more information about 
commutative group stacks we refer to~\cite[\S 1.4]{SGA4} or~\cite{Brochard14}.

Let $\cD^{[-1,0]}(\bS)$ denote the subcategory of the derived category 
of abelian sheaves on $\bS$
 consisting of complexes $K$ such that ${\HH}^i (K)=0$ for $i \not= -1$ or~$0$. 
Denote by ${\bPicard}(\bS)$ the category whose objects are commutative group 
stacks and whose arrows are isomorphism classes of additive functors. In~\cite[ 
\S 1.4]{SGA4} Deligne constructs an equivalence of categories
\begin{equation}
\label{st}
 \st: \cD^{[-1,0]}(\bS) \longrightarrow  {\bPicard}(\bS).
\end{equation}
We denote by $[\,\,]$ the inverse equivalence of $\st$. Via this equivalence of 
categories to each 1-motive $M$ is associated a commutative group $S$-stack $\st(M)$ 
and morphisms of 1-motives correspond to additive functors between the 
corresponding 
commutative group stacks. 

We will denote by $B\GG_m$ the classifying $\bS$-stack of $\GG_m$, i.e. the 
commutative group $\bS$-stack such that for any object $U$ of $\bS$, 
$B\GG_m(U)$ is 
the category of $\GG_m$-torsors over $U$. Remark that $[B\GG_m]=\GG_m [1]$ 
where $\GG_m [1]$ is the complex with the multiplicative sheaf $\GG_m$ in degree 
-1. If $\cG$ and $\cQ$ are two commutative group stacks, 
${\cHom}(\cG,\cQ)$ will be the commutative group $\bS$-stack such that 
 for any object $U$ of $\bS$, ${\cHom}(\cG,\cQ)(U)$ is the category whose 
objects are additive functors from ${\cG}_{|U}$ to ${\cQ}_{|U}$ and whose 
arrows are morphisms of additive functors. We have that 
$ [{\cHom}(\cG,\cQ)] = \tau_{\leq 0}{\R}{\Hom}\big([\cG],[\cQ]\big) $, where 
$\tau_{\leq 0}$ is the good truncation in degree 0. \emph{The dual 
$D(\cG)$ of a commutative group stack} $\cG$ is the commutative group stack 
${\cHom}(\cG,B\GG_m)$. In particular $[D(\cG)] = \tau_{\leq 
0}{\R}{\Hom}\big([\cG],\GG_m[1]\big).$ Note that the Cartier duality 
of 1-motives coincides with the duality for commutative group stacks via the equivalence $\st$, i.e. 
$D(\st(M))\simeq \st(M^*)$, where $M^*$ is the Cartier dual of the 1-motive $M$ (see \cite[(10.2.11)]{D1}).

Let $S$ be an arbitrary scheme. An \emph{abelian $S$-scheme} $A$ is an $S$-group scheme which is smooth, proper 
over $S$ and with connected fibers. An \emph{$S$-torus} $\tore$ is an $S$-group 
scheme which is locally isomorphic for the fpqc topology (equivalently for the 
\'etale topology) to an $S$-group scheme of type ${\GG}_m^r$ (with~$r$ a 
nonnegative integer and ${\GG}_m^0$ the trivial torus).
If $G$ is an $S$-group scheme, we denote by $G^D$ the $S$-group scheme 
${\uHom}(G,{\GG}_m)$ of group homomorphisms from $G$ to $\Gm$. 
If $T$ is an $S$-torus, then 
$\tore^D$ is an $S$-group scheme which is locally
for the \'etale topology a constant group scheme defined by a finitely 
generated free $\ZZ$-module.


\section{Line bundles on 1-motives}
\label{section:line_bundles}

Let $S$ be a scheme. The following definition is directly inspired from~\cite[p. 
64]{Mumford65}.

\begin{definition}\label{def-linebundle-stack}
Let $p : \cX \to S$ be an $S$-stack.
\begin{enumerate}
 \item A \emph{line bundle} $\cL$ on $\cX$ consists of 
\begin{itemize}
	\item for any $S$-scheme $U$ and any object $x$ of $\cX(U)$, 
a line bundle $\cL(x)$ on $U$;
	\item for any arrow $f:y \rightarrow x$ in $\cX$, an isomorphism 
$\cL(f): \cL(y) \rightarrow p(f)^*\cL(x)$
	of line bundles on $U$ verifying the following compatibility: if $f:y 
\rightarrow x$ and $g:z \rightarrow y$ are two arrows of $\cX$, then $\cL(f 
\circ g) = p(g)^* \cL(f) \circ \cL(g)$.
\end{itemize}
\item A \emph{morphism} $F : 
\cL_1 \rightarrow \cL_2$ of line bundles over $\cX$ consists of a morphism of 
line bundles $F(x): \cL_1(x) \rightarrow \cL_2(x)$ for any $S$-scheme $U$ 
and for any object $x$ of $\cX(U)$,
such that $p(f)^*F(x)\circ \cL_1(f)= \cL_2(f)\circ F(y)$
for any arrow $f : y \rightarrow x$ in $\cX$.
\end{enumerate}
\end{definition}

\noindent
The usual tensor product of line bundles over schemes 
extends to stacks and allows us to define the tensor product $\cL_1 \otimes \cL_2$
 of two line bundles $\cL_1$ and $\cL_2$ on the stack $\cX$. 
This tensor product 
equips the set of isomorphism classes of line bundles on $\cX$ with an abelian 
group law.
Using the equivalence of categories~\cite[ \S 1.4]{SGA4} between 1-motives 
and commutative group stacks, we can then define line bundles on 1-motives as 
follows.

\begin{definition}\label{def-linebundle-1-motive}
Let $M$ be a 1-motive defined over $S$.
\begin{enumerate}
	\item The category $\PIC(M)$ of line bundles on $M$ is the category of 
line bundles on $\st(M)$.
  \item The \emph{Picard group of} $M$, denoted by $\Pic(M)$, is the group of 
isomorphism classes of line bundles on $\st(M)$.
\end{enumerate}
\end{definition}

The following lemma will allow us to describe line bundles on a 1-motive $M=[X 
\stackrel{u}{\to} G]$ as 
line bundles on $G$ endowed with an action of $X$ that is compatible 
with the translation action of $X$ on $G$.

\begin{lemma}\label{lem:equiv-def}
Let $\iota : \cX_0 \to \cX$ be a representable morphism of stacks over $S$. 
Assume 
that $\iota$ is faithfully flat, and quasi-compact or locally of finite 
presentation. Then the category of line bundles on $\cX$ is equivalent to the 
category of line bundles on $\cX_0$ with descent data, that is to the category 
whose objects are pairs $(\cL,\delta)$ where $\cL$ is a line 
bundle on $\cX_0$ and $\delta : q_1^*\cL \to q_2^*\cL$ is an isomorphism such 
that, up to canonical isomorphisms, $p_{13}^*\delta=p_{23}^*\delta\circ 
p_{12}^*\delta$ (with the obvious notations for the 
projections $q_i : \cX_0 \times_{\cX} \cX_0 \to \cX_0$ and $p_{ij} : \cX_0 
\times_{\cX} \cX_0 \times_{\cX} \cX_0 \to \cX_0 \times_{\cX} \cX_0$).
\end{lemma}
\begin{proof}
We have to prove that the pullback functor~$\iota^*$ from the category of line 
bundles on~$\cX$ to the category of line bundles on $\cX_0$ with descent data 
is an equivalence. The result is well-known if $\cX$ is algebraic, see 
\cite[(13.5)]{LMB}. Hence, for 
any $S$-scheme $U$ and any morphism $x : U\to \cX$, the statement is known for 
the morphism $\iota_U : \cX_0 \times_{\cX} U \to U$ obtained by base 
change. Since a line bundle 
on $\cX$ is by definition a collection of line bundles on the various schemes 
$U$, the general case follows. 
\end{proof}

Let $M=[X \stackrel{u}{\to} G]$ be a 1-motive over a scheme $S$.
 By \cite[(3.4.3)]{LMB} the associated commutative group stack $\st(M)$ is 
isomorphic to the quotient stack $[G/X]$ (where $X$ acts on~$G$ via the given 
morphism $u : X\to G$). Note that in general it is not 
algebraic in the sense of \cite{LMB} because it is not quasi-separated. However 
the quotient map $\iota : G \to [G/X]$ is representable, étale and surjective, 
and the above lemma applies. The fiber product $G\times_{[G/X]}G$ is 
isomorphic to $X\times_S G$. Via this identification, the projections 
$q_i:G\times_{[G/X]}G \to G $ (for $i=1,2$)
correspond respectively to the second projection $p_2: X\times_S G \to G $ 
and to the map $\mu: X\times_S G \to G$ given by the action $(x,g)\mapsto 
u(x)g$. We can further identify the fiber 
product $G\times_{[G/X]}G\times_{[G/X]}G$ with $X\times_S X\times_S G$ and the 
partial projections $p_{13}, p_{23}, p_{12}: G\times_{[G/X]}G\times_{[G/X]}G 
\to G\times_{[G/X]}G$ 
respectively with the map $m_X\times \id_G :  X\times_S X \times_S G \to 
X\times_S G$ where $m_X$ denotes the group law of $X$, the map $\id_X\times \mu 
: X\times_S X \times_S G \to X\times_S 
G$, and the partial projection $p_{23}': X\times_S X \times_S G \to X\times_S 
G$.
 Hence by Lemma \ref{lem:equiv-def} the category of line 
bundles on $M$ is equivalent to the category of couples 
$$(L, \delta)$$ 
where $L$ is a line bundle on $G$ and $\delta$ is a 
descent datum for $L$ with respect to $\iota : G\to [G/X]$.
More explicitly, the descent datum $\delta$ is an isomorphism 
$\delta: p_2^* L \rightarrow \mu^*L$ of line bundles on~$X\times_S G$
satisfying the cocycle condition
$$(m_X\times id_G)^* \delta=\big((\id_X\times
\mu)^*\delta\big) \circ \big((p_{23}')^* \delta\big).$$
It is often convenient to describe line bundles in terms of ``points''. If $g$ 
is a point of $G$, i.e. a morphism $g : U\to G$ for some $S$-scheme $U$, we 
denote by $L_g$ the line bundle $g^*L$ on $U$. Then~$\delta$ is given by a 
collection of isomorphisms
\[
 \delta_{x,g} : L_g \to L_{u(x)g}
\]
for all points $x$ of $X$ and $g$ of $G$, such that
for all points $x,y$ of $X$ and $g$ of $G$,
\begin{equation}
\label{eq:CC}
 \delta_{x+y,g}=\delta_{x,u(y)g} \circ \delta_{y,g}\, .
\end{equation}
With this description, the pullback functor $\iota^*$ maps a
line bundle $(L, \delta)$ on $M$ to $L$, i.e. $\iota^*$ just forgets
the descent datum. Note for further use that $\iota^*$ is faithful.


\section{D\'evissage of the Picard group of a 1-motive}
\label{section:devissage}

Let us first recall the following global version of 
Rosenlicht's Lemma from \cite[Corollaire VII 
1.2]{Raynaud_Faisceaux_amples_sur}.

\begin{lemma}[Rosenlicht]
\label{lemma:invertible_sections_on_a_sav}
 Let $S$ be a reduced base scheme and let $P$ be a flat 
$S$-group scheme locally of finite presentation. 
Assume that the maximal fibers of $P$ are smooth and 
connected. Let $\lambda : P \to \GG_m$ be a morphism of $S$-schemes.
If $\lambda(1)=1$, then $\lambda$ is a 
group homomorphism.
\end{lemma}

\noindent
\textbf{(I) First d\'evissage coming from the short exact sequence $0 
\rightarrow T \stackrel{i}{\rightarrow} G \stackrel{\pi}{\rightarrow} A 
\rightarrow 0.$}

\begin{proof}[Proof of Proposition \ref{prop:TGA}]
By \cite[Chp I, 
Prop 7.2.2]{Moret_Bailly_Pinceaux}, the category $\CUB(A)$ 
is equivalent to the category of pairs $(L,s)$ where $L$ is 
a cubical line bundle on $G$ and $s$ is a trivialization 
of $i^*L$ in the category $\CUB(T)$. With this 
identification, the pullback functor $\pi^* : \CUB(A) \to \CUB(G)$ is 
the forgetful functor that maps a pair $(L,s)$ to $L$.  But 
since the base scheme is assumed to be normal, all these 
categories of cubical line bundles are equivalent to the 
categories of line bundles rigidified along the unit section 
\cite[Chp I, Prop 2.6]{Moret_Bailly_Pinceaux}. The group of 
isomorphism classes of rigidified line bundles on $G$ is 
isomorphic to $\Pic(G)/\Pic(S)$, and similarly for $A$ 
and~$T$. Hence the equivalence of categories \cite[Chp I, 
Prop 7.2.2]{Moret_Bailly_Pinceaux} induces the following exact 
sequence when we take the groups of isomorphism classes:
\begin{equation}
\begin{array}{c} \displaystyle
\Aut(\cO_G) \stackrel{i^*}{\longrightarrow}  
\Aut(i^*\cO_G) \longrightarrow
\frac{\Pic(A)}{\Pic(S)} \stackrel{\pi^*}{\longrightarrow}  
\frac{\Pic (G)}{\Pic(S)} \stackrel{i^*}{\longrightarrow} 
\frac{\Pic(T)}{\Pic(S)}\, ,\\
\end{array}
\label{eq:devissage1}
\end{equation}
where the automorphism groups on the left are the 
automorphism groups in the categories of rigidified line 
bundles on $G$ and on $T$. An automorphism of $\cO_G$ 
(rigidified) is an automorphism $\lambda : \cO_G\to \cO_G$ 
such that $e^*\lambda=\id$ where $e$ is the unit section of 
$G$. Hence the above group $\Aut(\cO_G)$ identifies with 
the kernel of $e^* : \Gamma(G,\cO_G^*)\to \Gamma(S, 
\cO_S^*)$, i.e. with the group of morphisms of schemes $\lambda : G\to \Gm$ 
such that $\lambda(1)=1$. Since $S$ is reduced, this kernel is 
isomorphic to $\Hom(G, \GG_m)$ by 
Lemma~\ref{lemma:invertible_sections_on_a_sav}. Similarly,
the group $\Aut(i^*\cO_G)$ of automorphisms in the category 
of rigidified line bundles is isomorphic to $\Hom(T, 
\GG_m)$. Moreover, since $\Hom(A, \GG_m)=0$ the first map 
$i^*$ is injective.
\end{proof}

\begin{remark}
\label{rem:vanishing_Pic(T)}
(1) 
Over any base scheme $S$, by 
\cite[Chp I, Prop 7.2.1]{Moret_Bailly_Pinceaux} the category $\CUB(T)$ is 
isomorphic to the category of extensions of $T$ by $\GG_m$.
Moreover, by \cite[Chp I, Remark 7.2.4]{Moret_Bailly_Pinceaux}, if we assume 
the base scheme $S$ to be normal, or geometrically unibranched, or local 
henselien, then the group $\Ext^1(T, \GG_m)$ vanishes if the torus $T$ is split.

(2) If $L$ is a rigidified line bundle on $G$, the class of the line bundle
 $i^*L$ in $\Pic(T) / \Pic(S)$ represents the 
obstruction to the fact that $L$ comes from a rigidified line bundle over~$A$. 
Since $\Pic(T) / \Pic(S)\simeq \Ext^1(T, \GG_m)$ and since the tori underlying 
1-motives are split locally for the \'etale topology, as a consequence of 
(1) of this Remark we have that if $S$ is normal,
 there exists an \'etale and surjective morphism $S'\to S$ such that  
$(i^*L)_{|_{S'}}=0$, i.e. after a base change to $S'$, the rigidified line 
bundle $L$ on $G$ comes from $A$.
\end{remark}

\noindent 
\textbf{(II) Second d\'evissage coming from the exact sequence $0 
\rightarrow G \stackrel{\iota}{\rightarrow} M \stackrel{\beta}{\rightarrow} 
X[1] 
\rightarrow 0.$} \bigskip

Let us describe
more explicitly the maps $\iota^* : \Pic(M) \to \Pic(G) $ and $\beta^*:  \Pic(X[1]) \to \Pic(M)$ in terms of line
bundles with descent data. As explained in \S\ref{section:line_bundles}, we 
identify the category of line bundles on $M$
with the category of couples 
$$(L, \delta)$$
where $L$ is a line bundle on $G$ and $\delta$ is a descent
datum for $L$ with respect to the covering $\iota: G\to [G/X]$.
Then the pullback functor~$\iota^*$ maps a
line bundle $(L, \delta)$ on $M$ to $L$: $\iota^*(L, \delta) =  L.$

If $L$ is the trivial bundle $\cO_G$, via the canonical 
isomorphism $p_2^*L\simeq \mu^*L$, a descent datum $\delta$ on $L$ can 
be seen as a
morphism of $S$-schemes $\delta : X\times_S G \to \GG_m$, and the
cocycle condition (\ref{eq:CC}) on $\delta$ can be rewritten as follows:
for any points $x,y$ of $X$ and $g$ of $G$, we have the equation
\begin{equation}
\label{CC}
 \delta(x+y,g)=\delta(x,
u(y)g).\delta(y , g)\, .
\end{equation}

The category of line bundles on $X[1]$ is equivalent
to the category of line bundles on $S$ together with a
descent datum with respect to the presentation $S\to [S/X]$. 
By \cite[Example 5.3.7]{Brochard09} we have that 
$$ \frac{\Pic(X[1])}{\Pic(S)}\simeq \Hom(X, \GG_m). $$ 

Let us now describe the pullback morphism
$\beta^*$ in these terms.
Unwinding the various definitions, it can be seen that given
a character $\alpha : X\to \GG_m$,  the associated
element $\beta^*\alpha \in \Pic(M)$ is the class of the
line bundle $(\cO_G, \delta_{\alpha})$ where
$\delta_{\alpha}$ is the automorphism of $\cO_{X\times_S G}$
corresponding to the morphism of $S$-schemes 
$\delta_{\alpha} : X\times_S G\to \GG_m,\,\,  (x,g)  \mapsto \alpha(x)$:
$$ \beta^*\alpha =[(\cO_G, \delta_{\alpha})].$$ 

Even if the composition $\iota^*\beta^*$
is trivial, the sequence
\[\Pic(X[1]) \to \Pic(M) \to \Pic(G)
\]
is not exact in 
general as shown in the 
following example. However, in the special case of 1-motives without toric part, this sequence is always exact (see Remark \ref{remark:cas:A}).

\begin{example}
\label{sequence_not_exact}
Let $S$ be any base scheme with $\Pic(S)=0$. Let $T$ be an
$S$-torus, let $X=\ZZ$ and let $M=[u : X \to T]$ be a 1-motive with $u$ the 
trivial
morphism. Let $(\cO_T, \delta)$ be a line bundle on $M$
(using the above description) that is mapped to the neutral element of 
$\Pic(T)$. Note that since $u$ is trivial the cocycle condition \eqref{CC}
here means that for any $g\in T(U)$, $\delta(.,g)$
is a group homomorphism in the variable $x$. 

The class
of $(\cO_T, \delta)$ is in the image of $\Pic(X[1])$
if and only if there is an $\alpha\in \Hom(X,\GG_m)$
such that $(\cO_T, \delta)\simeq (\cO_T, \delta_{\alpha})$.
An isomorphism
$(\cO_T, \delta)\simeq (\cO_T, \delta_{\alpha})$
is an automorphism $\lambda$ of $\cO_T$ such that
$\delta_{\alpha}\circ p_2^*\lambda=\mu^*\lambda\circ \delta$.
But here $\mu=p_2$ (since $u$ is trivial) and the group
of automorphisms of $\cO_{X\times_S T}$ is commutative.
So $(\cO_T, \delta)$ and $(\cO_T, \delta_{\alpha})$
are isomorphic if and only if $\delta=\delta_{\alpha}$.
This proves that $(\cO_T, \delta)$ is in the image of
$\Pic(X[1])$ if and only if $\delta$, seen as a morphism of $S$-schemes
$\delta : X\times_S T \to \GG_m$, is constant
in the variable $g\in T$ (for the ``if'' part, we define
$\alpha$ by $\alpha(x)=\delta(x,1)$ and the cocycle condition on $\delta$ 
ensures that
$\alpha$ is a group homomorphism). We will now construct
a descent datum $\delta$ on $\cO_T$ which is not constant
in $g$ and this will prove that the sequence
$\Pic(X[1]) \to \Pic(M) \to \Pic(T)$ is not exact.
Let $\lambda \in \Hom(T, \GG_m)$ be a non trivial
homomorphism and define $\delta$ functorially by
$\delta(n,g)=\lambda(g)^n$. This $\delta$ is a homomorphism
in the variable $n$ for any $g$ and so it is indeed a descent
datum, but it is non constant in $g$ since $\lambda$
is non constant. Hence the corresponding line bundle $(\cO_T, \delta)$
is not in the image of $\Pic(X[1])$.  
\end{example}

Now we compute the kernel of $\iota^*: \Pic(M)\to \Pic(G).$
Denote by $G^{D} = \underline{\Hom}(G, \GG_m)$ and $X^{D} = \underline{\Hom}(X, 
\GG_m)$ the Cartier duals of $G$ and $X$, respectively. (Note that calling 
$G^D$ the ``Cartier dual'' of $G$ is a slight abuse here since $G^{DD}$ does 
not need to be isomorphic to $G$. For instance if $G$ is an abelian scheme then 
$G^D=0$.)

\begin{lemma}\label{lem:equivalentconditions}
For a morphism of $S$-group schemes $\lambda : X\to G^{D}$, the 
following conditions are equivalent:
\begin{enumerate}
 \item For any $S$-scheme $U$ and any two points $x,y \in 
X(U)$, $\lambda(x)(u(y))=\lambda(y)(u(x))$.
\item The following diagram commutes
\begin{equation}\label{diagramme_equivalentconditions}
\xymatrix{
   X \ar[rr]^{\lambda} \ar[d]_u && G^{D} \ar[d]^{u^{D}} \\
   G \ar[r]^-{ev} & G^{DD} \ar[r]^{\lambda^{D}} & X^{D},
}
\end{equation}
where $ev : G \to (G^{D})^{D}$ is the canonical morphism that maps  $g\in 
G(U)$ to $ev_g : G^D \to \Gm, \varphi \mapsto \varphi(g)$, and  where 
$u^{D}$ (resp. $\lambda^{D}$) is the morphism of group schemes given by 
$\varphi 
\mapsto \varphi\circ u$ (resp. $\varphi \mapsto \varphi\circ \lambda$).
\end{enumerate}
\end{lemma}

\begin{proof} 
For any $S$-scheme $U$ and any $x,y\in X(U)$, we have
\begin{align*}
 u^D(\lambda(x))(y) &= (\lambda(x)\circ u)(y)\\
&= \lambda(x)(u(y))
\end{align*}
and
\begin{align*}
 (\lambda^D\circ ev\circ u)(x)(y) &= (\lambda^D(ev_{u(x)}))(y) \\
&= (ev_{u(x)}\circ \lambda) (y) \\
&=\lambda(y)(u(x))
\end{align*}
\end{proof}

We say that a morphism of $S$-schemes $\sigma : X\times_S X \to 
\GG_m$ is \emph{symmetric} if it satisfies the equation 
$\sigma(x,y)=\sigma(y,x)$. If $\alpha : X\to \GG_m$ is a 
morphism of $S$-schemes, we denote by $\sigma_{\alpha} : X\times_S X 
\to \GG_m$ the symmetric morphism given by 
$\sigma_{\alpha}(x,y)=\frac{\alpha(x+y)}{\alpha(x)\alpha(y)}$. 
Hence $\alpha$ is a morphism of $S$-group schemes if and only if 
$\sigma_{\alpha}$ is trivial. 

\begin{definition}\label{Lambda}
\begin{enumerate}
 \item We denote by $\Lambda$ the subgroup of $\Hom(X, G^{D})$ 
consisting 
of those morphisms of $S$-group schemes that satisfy the
equivalent conditions of Lemma \ref{lem:equivalentconditions}. 
\item We denote by $\Sigma$ the quotient of the group of 
symmetric bilinear morphisms $X\times_S X \to \GG_m$ by the subgroup 
of morphisms of the form $\sigma_{\alpha}$ for some morphism of
$S$-schemes~$\alpha : X \rightarrow \GG_m$.
\item We denote by $\Psi: \Lambda \to\Sigma$ the natural homomorphism 
that maps a $\lambda\in \Lambda$ to the class of the function $(x,y)\mapsto 
\lambda(x)(u(y))$.
\end{enumerate}
\end{definition}

\begin{remark}\label{Sigma}
Note that, following \cite[XIV, \S2 to \S4]{Cartan_Eilenberg} we can view 
$\Sigma$ as a subgroup of the kernel of the natural morphism $\Ext^1(X, \Gm) 
\to H^1(X,\Gm)$. Since the framework and statements of \cite{Cartan_Eilenberg} 
are not exactly the same as ours, we briefly recall the construction here. If 
$\sigma : X\times_S X \to \Gm$ is a symmetric bilinear morphism, let 
$E_{\sigma}$ be the group scheme $\Gm\times_S X$, where the group law is given 
by $(\gamma_1, x).(\gamma_2,y):=(\gamma_1\gamma_2\sigma(x,y), x+y)$. With the 
second projection $\pi : E_{\sigma} \to X$ and the inclusion $i : \Gm \to 
E_{\sigma}$ given by $i(\gamma)=(\gamma,0)$, the group scheme $E_{\sigma}$ is a 
commutative extension of $X$ by $\Gm$. Then a direct computation shows that 
$\sigma \mapsto E_{\sigma}$ induces an injective group homomorphism from 
$\Sigma$ to $\Ext^1(X,\Gm)$. Since the projection $\pi : E_{\sigma} \to X$ has 
a section $x\mapsto (1,x)$, the $\Gm$-torsor over $X$ induced by $E_{\sigma}$ 
is trivial, which proves that the image of $\Sigma$ lies in the kernel of 
$\Ext^1(X,\Gm)\to H^1(X,\Gm)$. Actually, if $E$ is an extension of $X$ by 
$\Gm$, its class $[E]\in \Ext^1(X,\Gm)$ lies in $\Sigma$ if and only if the 
projection $E\to X$ has a section $s : X\to E$ (only as a morphism of schemes, 
not of group schemes) which is of degree 2 in the language 
of~\cite{Breen_Fonctions_Theta} or \cite{Moret_Bailly_Pinceaux}, i.e. such that 
$\theta_3(s)=1$.
\end{remark}

\begin{remark}
 \label{annulation_Sigma}
In particular, if $X$ is split (that is, $X\simeq \ZZ^r$ for some $r$) then 
$\Sigma=0$ since the morphism $\Ext^1(X, \Gm) 
\to H^1(X,\Gm)$ is injective.
\end{remark}

For the rest of this Section, we assume that the base scheme~$S$ is reduced. 
Denote by~$K$ the kernel of the forgetful functor $\iota^*: \Pic(M) \to 
\Pic(G)$. This kernel is the group of classes of 
pairs $(\cO_G, \delta)$, where $\delta$ is a descent datum on 
$\cO_G$. Such a descent datum can be seen as a morphism of schemes $\delta : 
X\times_S G\to \GG_m$ that satisfies the cocycle condition \eqref{CC}. Two 
pairs 
$(\cO_G, 
\delta_1), (\cO_G,\delta_2)$ are in the same class if and only if 
they are isomorphic in the category of 
line bundles on $G$ equipped with a descent datum relative 
to $\iota : G \to M$, which means that there is a 
morphism of $S$-schemes $\nu : G\to \GG_m$ such that
$(\mu^*\nu).\delta_1=\delta_2.p_2^*\nu$
where $\mu, p_2 : X\times_SG\to G$ are the action of $X$ on $G$
and the second projection. The latter equation can be rewritten as
$\nu(u(x)g)\delta_1(x,g)=\delta_2(x,g)\nu(g)$
for any $(x,g) \in X(U)\times G(U)$. Replacing $\nu$ with 
 $g\mapsto \nu(g)/\nu(1)$, we may 
assume that $\nu(1)=1$ so that $\nu$ is a group 
homomorphism by Rosenlicht's Lemma \ref{lemma:invertible_sections_on_a_sav}. 
The equation then 
becomes 
\begin{equation}
\nu(u(x))\delta_1(x,g)=\delta_2(x,g).
\label{eq:descriptionK-1}
\end{equation}
 The group law on $K$ is given 
by $[(\cO_G, \delta_1)].[(\cO_G, 
\delta_2)]=[(\cO_G, \delta_1.\delta_2)]$.

We will now construct a homomorphism  
$\Theta: K\to \Lambda$, where $\Lambda$ was defined in Definition \ref{Lambda}. Let $[(\cO_G, \delta)]$ be a class in $K$ where $\delta$
is a solution of \eqref{CC}. For any point 
$x$ of $X$, consider the morphism of $S$-schemes  
\begin{equation}
\lambda_\delta(x) : G\to \GG_m, \,\, g \mapsto 
\frac{\delta(x,g)}{\delta(x,1)}.
\label{eq:descriptionK-2}
\end{equation}
 Since 
$\lambda_\delta(x)(1)=1$, the morphism $\lambda_\delta(x)$ is actually a 
homomorphism by Lemma \ref{lemma:invertible_sections_on_a_sav}, hence a section 
of 
$G^{D}$. This construction is functorial and defines a morphism of 
$S$-schemes 
$\lambda_\delta : X\to G^{D}$. By \eqref{CC}, for any $x,y\in X$ and any 
$g\in G$ 
we 
have
\begin{align*}
 \lambda_\delta(x+y)(g) &= \frac{\delta(x+y,g)}{\delta(x+y,1)} \\
&= 
\frac{\delta(x,u(y)g)\delta(y,g)}{\delta(x,u(y))\delta(y,1)}
\\
&=\frac{\delta(x,u(y)g)}{\delta(x,1)}.
\frac{\delta(x,1)}{\delta(x,u(y))}.
\frac{\delta(y,g)}{\delta(y,1)} \\
&= \frac{\lambda_\delta(x)(u(y)g)}{\lambda_\delta(x)(u(y))}. 
\lambda_\delta(y)(g) \\
&= \lambda_\delta(x)(g).\lambda_\delta(y)(g)
\end{align*}
where the last equality follows from the fact that $\lambda_\delta(x)$ is a
 homomorphism. Hence $\lambda_\delta$ is a morphism of $S$-group schemes. 
Moreover, by \eqref{CC} 
for any $x,y\in X$ we have
\[
 \delta(x,u(y))\delta(y,1)=\delta(x+y,1)=\delta(y+x,1)
=\delta(y,u(x))\delta(x,1).
\]
Hence $\lambda_\delta(x)(u(y))=\lambda_\delta(y)(u(x))$ and so 
$\lambda_\delta$ belongs to $\Lambda$. Since $\lambda_{\delta}$ only 
depends on the class $[(\cO_G, \delta)]$, this construction induces 
a well-defined homomorphism 
\begin{equation}
\Theta: K\to \Lambda, \,\, [(\cO_G, \delta)] \mapsto \lambda_\delta.
\label{eq:Theta}
\end{equation}
 It is a homomorphism because 
$\lambda_{\delta_1\delta_2}=\lambda_{\delta_1}\lambda_{\delta_2}$.

\begin{proof}[Proof of Proposition \ref{prop:devissage_kernel}]
  The morphism $\beta^*: \Hom(X, \GG_m) \rightarrow K$ maps an $\alpha \in 
\Hom(X, \GG_m)$ to the 
class $[(\cO_G, \delta_{\alpha})]$, where $\delta_{\alpha}$ is defined by 
$\delta_{\alpha}(x,g)=\alpha(x)$. By the equality \eqref{eq:descriptionK-1}, 
$[(\cO_G, 
\delta_{\alpha})]$ is trivial if and only if there is a morphism of $S$-group 
schemes 
$\nu : G\to \GG_m$ such that $\alpha=\nu\circ u$, which means that 
the sequence is exact in $\Hom(X,\GG_m)$.

Now we check the exactness in $K$. 
Let $[(\cO_G, \delta)]$ be a class in $K$.
By \eqref{eq:descriptionK-2} its image $\lambda_{\delta}$ under~$\Theta$ 
is trivial if and only if $\delta$ satisfies the equation 
$\delta(x,1)=\delta(x,g)$ for any $x\in X$ and~$g\in G$. If so, let $\alpha : 
X\to \GG_m$ be the morphism of $S$-schemes defined by $\alpha(x)=\delta(x,1)$. 
Then by \eqref{CC} 
$\alpha$ is a homomorphism, and we have $\delta=\delta_{\alpha} = 
\beta^*(\alpha)$, which 
proves the exactness in~$K$.

It remains to prove the exactness in $\Lambda$. Let $\lambda\in \Lambda$. 
Assume that $\lambda$ is in the image of~$K$, i.e. there is some solution 
$\delta$ of \eqref{CC} such that $\lambda=\lambda_{\delta}$. Let $\alpha : 
X\to 
\GG_m$ be the morphism of $S$-schemes defined by $\alpha(x)=\delta(x,1)$. Then 
for 
any $(x,g) \in X \times G$ we have $\delta(x,g)=\lambda(x)(g)\alpha(x)$. 
The bilinearity of $\lambda$ and \eqref{CC} yield
$\lambda(x)(u(y))=\frac{\alpha(x+y)}{\alpha(x)\alpha(y)}$. Hence the image of 
$\lambda$ in $\Sigma$ is trivial. Conversely, assume that the image 
$\Psi(\lambda)$
is trivial in $\Sigma$, in other words there is a morphism of $S$-schemes 
$\alpha : X \to \GG_m$ 
such that $\lambda(x)(u(y))=\frac{\alpha(x+y)}{\alpha(x)\alpha(y)}$. Then we 
define $\delta$ by $\delta(x,g)=\lambda(x)(g)\alpha(x)$ and the same 
computations as above show that $\delta$ satisfies \eqref{CC} and that 
$\lambda=\lambda_{\delta}$, which concludes the proof.
\end{proof}

 If the lattice $X$ underlying the 1-motive $M=[u:X \to G]$ is split then by 
Remark~\ref{annulation_Sigma} the morphism $K\to 
\Lambda$ is surjective. 
Actually we can give an explicit section, that depends on the choice of a 
$\ZZ$-basis for $X$, as follows. Let $e_1, \dots, e_n$ be a $\ZZ$-basis of $X$.
For $\lambda\in \Lambda$, let $\lambda_1, \dots, \lambda_l : G\to \Gm$ 
be the images of $e_1, \dots, e_l$ under $\lambda$. We denote 
by~$\delta_{\lambda}$ the morphism from $X\times_S G$ to $\GG_m$ defined by
\begin{equation}
 \label{def:varphi_lambda}
\delta_{\lambda}(x,g)=\lambda(x)(g)\prod_i 
\left(\lambda_i\circ 
u\left(\frac{n_i(n_i-1)}{2}e_i\right)\right) \prod_{1\leq 
i<j\leq l}
\lambda_i(u(e_j))^{n_in_j}\, .
\end{equation}
for any $S$-scheme $U$, any $x=\sum n_ie_i\in X(U)$ and any 
$g\in G(U)$.

\begin{proposition}\label{prop:devissage_PicM}
Let $M=[u:X \to G]$ be a 1-motive defined over a reduced base scheme $S$.
Assume that the lattice $X$ is split. With the above notations, the 
application $\lambda \mapsto [(\cO_G,\delta_{\lambda})]$ defines a section $s 
: \Lambda \to K$ of the homomorphism $\Theta$ defined in (\ref{eq:Theta}). In particular the group 
$\Pic(M)$ fits in the following exact 
sequence:
\begin{equation}
\Hom(G, \GG_m)
\longrightarrow \Hom(X, \GG_m) \times \Lambda
\longrightarrow \Pic (M)
\stackrel{\iota^*}{\longrightarrow}  \Pic (G)\, .
\label{eq:last}
\end{equation}
\end{proposition}
\begin{proof}
 A direct computation shows that 
$\delta_{\lambda}$ satisfies the equation 
\eqref{CC}, hence it is a descent datum and $s$ is well-defined. From the 
definition of $\delta_{\lambda}$, we see that $\delta_{\lambda.\lambda'}=
\delta_{\lambda}.
\delta_{\lambda'}$ hence $s$ is a group homomorphism. Moreover, the quotient 
$\delta_{\lambda}(x,g)/\delta_{\lambda}(x,1)$ is equal to 
$\lambda(x)(g)$, which proves that 
$\Theta([(\cO_G,\delta_{\lambda})])=\lambda$. 
The 
exact sequence (\ref{eq:last}) now follows from Proposition~\ref{prop:devissage_kernel}.
\end{proof} 

\begin{remark}\label{remark:cas:A} Let $M=[v:X 
\rightarrow A]$ be a 1-motive without toric part.
Since $\underline{\Hom}(A,\GG_m)=0$, the group $\Lambda$ is trivial and so from 
Proposition \ref{prop:devissage_kernel}, we obtain that $ \beta^* :\Hom(X,\GG_m) 
\to K$ is an isomorphism, that is the short sequence defined by $\beta^*$ and $\iota^*$, 
 $  \Pic(X[1]) / \Pic(S) \rightarrow 
\Pic(M)/ \Pic(S) \rightarrow \Pic(A)/ \Pic(S) $, is exact.
\end{remark}

\section{Construction of $\Phi: \Pic(M)/\Pic(S) \to \Hom(M,M^*)$}
\label{section:direct_construction}

Using the d\'evissage of the Picard group of a 1-motive $M$, in this Section we construct the morphism 
$\Phi : \Pic(M)/\Pic(S) \to \Hom(M,M^*)$ of Theorem~\ref{thm:existence_Phi} in an explicit way.

We start proving the following lemma which might be well-known, but for which we were unable to 
find a convenient reference.
\begin{lemma}
 \label{lem:morphisme_du_milieu}
Let $S$ be a reduced base scheme. Consider the following commutative diagram of 
commutative $S$-group schemes 
\[
\xymatrix{
0 \ar[r] &
T    \ar[r]^{i} \ar[d]_h&
G    \ar[r]^{\pi} \ar@{.>}[d]^u&
A    \ar[r] \ar[d]^v& 0 \\
0 \ar[r] &
T'    \ar[r]^{i'} &
G'    \ar[r]^{\pi'} &
A'    \ar[r] & 0
}
\]
where $T, T'$ are tori, $A, A'$ are abelian schemes, all the 
solid arrows are group homomorphisms, the rows are exact, and $u$ is only assumed to be a morphism of schemes over $S$. Then,
\begin{enumerate}
\item $u$ is a group homomorphism.
\item $u$ is uniquely determined by $h$ and $v$, i.e. if $u_1$ and 
$u_2$ are two morphisms that make the whole diagram commutative, then $u_1=u_2$.
 \item if $h=v=0$, then $u=0$.
\end{enumerate}
\end{lemma}
\begin{proof}
 Let us prove (3). Since $\pi'\circ u=0$ the morphism $u$ factorizes through a 
morphism of schemes $u' : G\to T'$. The question is local on $S$, and $T'$ is 
locally isomorphic to $\Gm^r$ for some integer $r$, hence we may assume that 
$T'=\Gm$. Since $u'\circ i$ is trivial, in particular $u'(1)=1$ and so
by Rosenlicht's Lemma~\ref{lemma:invertible_sections_on_a_sav} $u'$ is a group homomorphism. Now 
the result follows since $\Hom(A,\Gm)=0$.

Applying (3) with $u=u_1-u_2$ we get (2). Now let us prove (1). It 
suffices to apply~(2) with the exact sequence $0\to T\times_S T\to G\times_S 
G\to A\times_S A\to 0$ and the morphisms $u_1, u_2 : G\times_S G \to G'$ 
defined 
by $u_1(x,y)=u(x+y)$ and $u_2(x,y)=u(x)+u(y)$.
\end{proof}

Let $S$ be a normal base scheme and let $M=[u: X \rightarrow G]$ 
be a 1-motive over $S$, where $G$ fits in an 
extension $0 \rightarrow T \stackrel{i}{\rightarrow} G 
\stackrel{\pi}{\rightarrow} A 
\rightarrow 0.$  We start recalling from~\cite[(10.2.11)]{D1} the description of the Cartier dual
 $M^*=[u':T^D \rightarrow G']$ of $M$.
Denote by $\ov{M}$ the 1-motive $M/W_{-2}M= [v: X 
 \rightarrow A]$ where $v=\pi\circ u$. An extension of $\ov{M}$ by 
$\Gm$ is a pair $(E,\vt)$, where $E$ is an extension of $A$ by $\Gm$ and $\vt$ 
is a trivialization of $v^*E$:
\[
 \xymatrix{
&&& X\ar_{\vt}@/_1pc/[ld] \ar^v[d] &\\
0\ar[r]&\Gm \ar[r]& E \ar[r]&  A \ar[r]& 0}
\]
Extensions of $\ov{M}$ by $\Gm$ do not admit non trivial automorphisms. The 
functor of isomorphism classes of such extensions is representable by a group scheme
$G'$, which is an extension of $A^*$ by $X^D$:
\[
\xymatrix{
0 \ar[r] & X^D
    \ar[r]^{i'} & G'
    \ar[r]^{\pi'} & A^*
    \ar[r] & 0
}
\]
The 1-motive $M$ is an extension of $\ov{M}$ by $T$. If $\tau : T\to \Gm$ is a 
point of $T^D$, the pushdown $\tau_*M$ is an extension of $\ov{M}$ by $\Gm$, i.e. it is a point of $G'$. 
This defines a morphism $u' : T^D \to G'$ by $u'(\tau)=\tau_*M$ and by 
definition the Cartier dual of $M$ is the 1-motive $M^*=[T^D 
\stackrel{u'}{\rightarrow}G']$.

Now, we start the construction of $\Phi: \Pic(M)/\Pic(S) \to \Hom(M,M^*)$.
Let $(\cL,\delta)$ be a line bundle on $M$, where $\cL$ is a line bundle 
on $G$ and $\delta$ is a descent datum on $\cL$, i.e. an isomorphism 
\[\delta: p_2^* \cL \rightarrow \mu^*\cL\]
satisfying the cocycle condition (\ref{eq:CC}) (see end of 
Section~\ref{section:line_bundles}). We have to construct a morphism 
$\Phi(\cL,\delta) : M\to M^*$.
The first d\'evissage of $\Pic(M)$ (see Proposition \ref{prop:TGA}) furnishes the following 
exact 
sequence of groups
\[ 
\Hom(T,\GG_m) \stackrel{\xi}{\longrightarrow} 
\Pic(A)/\Pic(S) \stackrel{\pi^*}{\longrightarrow}  
\Pic(G)/\Pic(S) \stackrel{i^*}{\longrightarrow} 
\Pic(T)/\Pic(S).
\]
By Remark~\ref{rem:vanishing_Pic(T)} (2), 
since the tori underlying 1-motives are split locally for the \'etale topology, 
 there exists an \'etale and surjective morphism $S'\to S$ such that  
$(i^* \cL)_{|_{S'}}$ is trivial, which means that 
$$\cL_{|_{S'}} =\pi^* L$$
 for some line bundle $L \in \Pic(A_{|_{S'}}) /\Pic(S')$. 
Below we will construct locally defined linear morphisms $\Phi((\cL, 
\delta)_{|_{S'}}): M_{|_{S'}}\to 
M^*_{|_{S'}}$ from $M_{|_{S'}}$ to its Cartier dual $M^*_{|_{S'}}$. 
Since these are induced by a global line bundle $(\cL,\delta)$, they glue 
together and yield a linear morphism 
$\Phi(\cL,\delta) : M\to M^*$ over $S$. Hence it is not restrictive if we assume 
 $S'=S$ and $\cL=\pi^*L$ in order to simplify notation.

Via the classical homomorphism $\Phi_A: \Pic(A) \to \Hom(A,A^*)$, the line bundle $L$ furnishes 
a morphism of $S$-group schemes 
\[
\varphi_{L} :A \longrightarrow A^*, \quad a \mapsto 
(\mu_a^*L)\otimes L^{-1},
\]
 where $\mu_a : A\to A$ is the translation by $a$. Let us check that $\varphi_{L} :A \to A^*$ does not depend 
on the choice of the line bundle $L$ but only on its pullback $\cL=\pi^* L$, in 
other words $ \Phi_A \circ \xi =0$. Let $\alpha \in \Hom(T, \GG_m)$.
By definition of~$\xi$, 
$\xi(\alpha)$ is the image of the class $[\alpha_* G ]$
under the inclusion $ \Ext^1(A,\GG_m) \hookrightarrow \Pic (A)$, that 
is~$\xi(\alpha)$ 
comes from $\Ext^1(A,\GG_m)$. Hence by~\cite[Prop 1.8]{Raynaud_notes_de_cours} 
$\Phi_A(\xi(\alpha))=0$.

Our next aim is to define a morphism 
$\philt : G\to G'$ that lifts $\varphi_L$.
Before we recall briefly the isomorphism between $\underline{\mathrm{Ext}}^1(A,\GG_m)$ and $A^*:$
any extension of $A$ by $\GG_m$ is in particular a $\GG_m$-torsor over $A$ and therefore a line bundle over $A$, that is a point of $A^*$; on the other hand, to any line bundle $N$ over $A$ we associate the sheaf $E$ such that for any $S$-scheme $T$
\[E(T)=\{ (a, \tau) \; | \quad a \in A(T), \; \tau: N_T \stackrel{\cong}{\rightarrow}  \mu_a^* N_T\},
\]
where $N_T$ is the pull-back of $N$ to $A_T=A \times_S T$, which is in fact an 
extension of $A$ by $\GG_m$ (see \cite[\S2]{Demazure_notes_de_cours} for more 
details).
Now let $g\in G(S)$. The line bundle 
$\varphi_L(\pi(g))=\mu_{\pi(g)}^*L\otimes L^{-1}$ is a point of $A^*(S)$. We denote 
by  $E_{\varphi_L(\pi(g))}$ the corresponding extension of~$A$ by $\Gm$.
 As observed before, the extension 
$E_{\varphi_L(\pi(g))}$ has the following functorial description: 
$E_{\varphi_L(\pi(g))}(S)$ is the set of pairs $(a,\beta)$ where $a\in A(S)$ 
and 
$\beta : \varphi_L(\pi(g)) \to \mu_a^*\varphi_L(\pi(g))$ is an isomorphism of 
line bundles over $A$. We 
define functorially 
\begin{equation}
\label{eq:def_phiL_tilde}
\begin{array}{lrcl}
\philt : & G & \longrightarrow & G' \\
    & g & \longmapsto & \philt(g)=(E_{\varphi_L(\pi(g))}, \vt_{g}) ,
\end{array}
\end{equation}
 where the trivialization $\vt_g : X \to E_{\varphi_L(\pi(g))}$ is 
defined 
by
\begin{equation}
\label{eq:def_vtilde}
  \vt_g(x)=(v(x), \varphi_{g,x})
\end{equation}
with $\varphi_{g,x} : \varphi_L(\pi(g)) \to \mu_{v(x)}^*\varphi_L(\pi(g))$
 the isomorphism of line bundles on $A$ given by the following lemma.

\begin{lemma}
 \label{lem:def_vtilde}
With the above notation, there is a unique isomorphism $\varphi_{g,x} : 
\varphi_L(\pi(g)) \to \mu_{v(x)}^*\varphi_L(\pi(g))$ of line bundles on $A$ such that 
$\pi^*\varphi_{g,x} : \mu_g^*\cL\otimes \cL^{-1} \to
\mu_g^*(\mu_{u(x)}^*\cL)\otimes (\mu_{u(x)}^*\cL)^{-1}$ is equal to 
$\mu_g^*\delta_x\otimes \delta_x^{-1}$,
where $\delta_x : \cL \to \mu_{u(x)}^*\cL$ denotes the isomorphism 
$(x,\id_G)^*\delta$ of line bundles on $G$ induced by the descent datum~$\delta$.
\end{lemma}

\begin{proof}
For any $x\in X(S)$ and $b\in G(S)$, let us denote by 
$\ov{\delta}_{x,b}$ the isomorphism $\cO_{S} \to \cL_{u(x)b}\otimes\cL_b^{-1}$ 
induced by $\delta_{x,b}$ and by $\ov{\delta}_x: \cO_{G} \to \mu_{u(x)}^*\cL \otimes  \cL^{-1}$ the isomorphism induced by $\delta_x$. 
Consider the line bundle $N=\mu_{\pi(g)}^*(\mu_{v(x)}^*L\otimes 
L^{-1})\otimes (\mu_{v(x)}^*L\otimes L^{-1})^{-1}$ on $A$. In order to prove our Lemma it is enough to show that there is 
a unique isomorphism $\varphi : \cO_{A}\to N$ such that 
$\pi^*\varphi=\mu_g^*\ov{\delta}_x\otimes \ov{\delta}_x^{-1}$.

 By \cite[Chp I, Prop 2.6 and 7.2.2]{Moret_Bailly_Pinceaux} the pullback 
functor $\pi^*$ induces an equivalence between the 
category of rigidified (at the origin) line bundles on $A$, and the 
category of pairs $(\cN, \alpha)$ where $\cN$ is a rigidified line bundle on 
$G$ and $\alpha$ is a trivialization of $i^*\cN$ in the category of rigidified 
line bundles on~$T$. The line bundle $\cO_A$ is canonically rigidified at 1 and 
the line bundle $N$ on $A$ has a rigidification at 1 given by 
$\ov{\delta}_{x,g}\otimes \ov{\delta}_{x,1}^{-1}$. Hence by the above 
equivalence of categories to prove the Lemma it suffices to prove that 
$\mu_g^*\ov{\delta}_x\otimes \ov{\delta}_x^{-1}$ is compatible with the 
trivializations of $i^*\pi^*\cO_A$ and $i^*\pi^*N$. In other words, we 
have to prove that for 
any point $t$ of~$T$, the following diagram commutes:
\[\hskip-1cm
\xymatrix@R=0.4pc@C=0pc{
\cO_S \ar[r]^-{\ov{\delta}_{x,gi(t)}\otimes \ov{\delta}_{x,i(t)}^{-1}}
\ar[dd]_-{\ov{\delta}_{x,g}\otimes \ov{\delta}_{x,1}^{-1}} &
(\cL_{u(x)gi(t)}\otimes\cL_{gi(t)}^{-1})\otimes
(\cL_{u(x)i(t)} \otimes \cL_{i(t)}^{-1})^{-1}
\ar@{=}[d] \\
& (L_{\pi(u(x)gi(t))}\otimes L_{\pi(gi(t))}^{-1})\otimes
(L_{\pi(u(x)i(t))} \otimes L_{\pi(i(t))}^{-1})^{-1} \ar@{=}[d] \\
 (\cL_{u(x)g}\otimes\cL_{g}^{-1})\otimes
(\cL_{u(x)} \otimes \cL_{1}^{-1})^{-1} \ar@{=}[r]
& (L_{\pi(u(x)g)}\otimes L_{\pi(g)}^{-1})\otimes
(L_{\pi(u(x))} \otimes L_{1}^{-1})^{-1}
}
\]
This diagram defines an automorphism of $\cO_S$, hence an element of $\Gm(S)$, 
and the diagram commutes if and only if this element is equal to $1\in \Gm(S)$. 
As 
$g$ and 
$t$ vary, these diagrams induce a morphism of schemes $\zeta : G\times_S T \to 
\Gm$. If $t=1$, the diagram obviously commutes, hence $\zeta(g,1)=1$ and by 
Rosenlicht's Lemma~\ref{lemma:invertible_sections_on_a_sav} $\zeta(g,.)$ is a 
group homomorphism $T\to \Gm$. Then $\zeta$ corresponds to a morphism of 
schemes $G\to T^D$. Since $G$ has connected fibers and $T^D$ is a lattice, the 
latter morphism must be constant. But the diagram obviously commutes if 
$g=1$, hence $\zeta$ is constant equal to 1 and the diagram commutes for all 
points $g$ of $G$ and $t$ of $T$, as required.
\end{proof}

Now $\vt_{g}$ is well-defined and the formula~(\ref{eq:def_phiL_tilde})
defines a morphism of schemes $\philt : G\to G'$.
If $g\in G(S)$, the image $\pi'(\philt(g))$ is the class in $A^*(S)$ of the 
extension $E_{\varphi_L(\pi(g))}$, that is $\pi'(\philt(g))=\varphi_L(\pi(g))$,
and so the right-hand square in the following diagram is commutative. We 
denote by $h : T \to X^D$ the unique morphism that makes the left-hand square 
commutative:
\[
\xymatrix{
0 \ar[r] & T
    \ar[r]^{i} \ar@{.>}[d]^{h}& G
    \ar[r]^{\pi} \ar[d]^{\philt}& A
    \ar[r] \ar[d]^{\varphi_L} & 0 \\
0 \ar[r] & X^D
    \ar[r]_{i'} & G'
    \ar[r]_{\pi'} & A^*
    \ar[r] & 0
}
\]

\begin{remark}\rm
 \label{rem:description_h}
We can give an explicit description of $h : T\to X^D$ in terms of $(\cL, 
\delta)$ as follows. Let $t\in T(S)$ be a point of $T$. Then by definition 
$\philt(i(t))=(E_{\varphi_L(\pi(i(t)))}, \vt_{i(t)})$. Since~$\pi(i(t))=1$ the 
extension $E_{\varphi_L(\pi(i(t)))}$ is trivial. The morphism $h(t) : X\to \Gm$ 
is given by~$\vt_{i(t)}$. Let $x\in X(S)$. By definition 
$\vt_{i(t)}(x)=(v(x),\varphi_{i(t),x})$. Since the line bundle $\varphi_L(1)$ 
is trivial, the isomorphism $\varphi_{i(t),x} : \varphi_L(1)\to 
\mu_{v(x)}^*\varphi_L(1)$ can be seen as a morphism of schemes $A\to \Gm$, 
and $h(t)(x)\in \Gm(S)$ is the (necessarily constant) value of this morphism. 
We may evaluate it at the origin of $A$ and we see that $h(t)(x)$ is the point 
of $\Gm$ that corresponds to the isomorphism of (canonically trivial) line 
bundles $\delta_{x,i(t)}\otimes \delta_{x,1}^{-1} : \cL_{i(t)}\otimes 
\cL_{1}^{-1} \to \cL_{u(x)i(t)}\otimes \cL_{u(x)}^{-1}$.
\end{remark}

It is clear from the above Remark 
that $h$ does not depend on the choice of $L$. Moreover, since $h(1)=1$, it follows 
from Rosenlicht's Lemma~\ref{lemma:invertible_sections_on_a_sav} that $h$ is a group homomorphism. 
Then by Lemma~\ref{lem:morphisme_du_milieu} $\philt$ is also a group homomorphism, 
and it does not depend on the choice of the lifting~$L$ of~$\cL$ (since 
$\varphi_L$ does not depend on this choice as we have already proved).

The following proposition proves that the pair $(h^D, \philt)$ is a morphism of 
1-motives and so we can set
\[
\begin{array}{lrcl}
\Phi : & \Pic(M)/\Pic(S) & \longrightarrow & \Hom(M,M^*) \\
    & (\cL,\delta) & \longmapsto & \Phi(\cL,\delta)= (h^D, \philt)  .
\end{array}
\]

\begin{proposition}
 \label{prop:seconde_construction}
Let $h^D : X\to T^D$ be the Cartier dual of $h$. Then the diagram
\[
\xymatrix{
X \ar[r]^{h^D} \ar[d]_{u} &
T^D \ar[d]^{u'} \\
G \ar[r]_{\philt} & G'
}
\]
is commutative. In other words, the pair $(h^D, \philt)$ is a morphism of 
1-motives from $M$ to~$M^*$.
\end{proposition}
\begin{proof}
 Let $x\in X(S)$. We have to prove that $u'(h^D(x))=\philt(u(x))$. With the 
identification $X\simeq X^{DD}$, the morphism $h^D(x)$ is equal to $ev_x\circ h 
: T \to \Gm$ where $ev_x : X^D \to \Gm$ is the evaluation at $x$. Hence, by 
definition, $u'(h^D(x))$ is the extension of $\ov{M}$ by $\Gm$ obtained 
from~$M$ by pushdown along the morphism $ev_x\circ h$.
\begin{equation}
\label{eq:def_uprime}
 u'(h^D(x))={ev_x}_*h_*M
\end{equation}
Let $M_L= [\philt\circ u : X\to G']$ and $\ov{M_L} = M_L /W_{-2} M_L =[\varphi_L\circ v : X \to A^*]$. Consider the two morphisms of 1-motives $\varphi_L'=(id_X, \philt) : M\to M_L$ and 
$\ov{\varphi_L} =(id_X, \varphi_L) : \ov{M}\to \ov{M_L}$ which fit in the following diagram of extensions:
\[
\xymatrix{
0 \ar[r] &
  T  \ar[r]^{} \ar[d]^{h}&
  M  \ar[r]^{} \ar^{\varphi_L'}[d]&
  \ov{M}  \ar[r] \ar^{\ov{\varphi_L}}[d]& 0 \\
0 \ar[r] &
  X^D  \ar[r]^{} &
  M_L  \ar[r]^{} &
  \ov{M_L}  \ar[r] & 0
}
\]
By \cite[Chp VII, (7) and (8)]{Serre_groupesalgebriquesetcorpsdeclasses} the 
existence of $\varphi_L'$ proves that $h_*M$ and $\ov{\varphi_L}^*M_L$ are isomorphic as extensions 
of $\ov{M}$ by $X^D$. Combining this with \eqref{eq:def_uprime} we get that
\begin{equation}
 u'(h^D(x))={ev_x}_*\ov{\varphi_L}^*M_L
\end{equation}
We can describe extensions of $\ov{M_L}$ by $X^D$ in terms of pairs $(E, 
\xi)$ where $E$ is an extension of~$A^*$ by $X^D$ and $\xi$ is a trivialization 
of $(\varphi_L\circ v)^*E$. In these terms, the extension $M_L$ corresponds to 
$G'$ together with the morphism $\philt\circ u : X \to G'$. Hence the extension 
$\ov{\varphi_L}^*M_L$ of $\ov{M}$ by~$X^D$ corresponds to the pair 
$(\varphi_L^*G', \ov{v}),$ where the trivialization $\ov{v}$ is the morphism $X \to \varphi_L^*G'$ 
induced by $\philt\circ u$, with $\philt $ defined in (\ref{eq:def_phiL_tilde}):
\[
\xymatrix{
&&& X \ar@/_1pc/[ld]_{\ov{v}} \ar[d]^v\\
0 \ar[r] &
X^D    \ar[r]^{} \ar@{=}[d]&
\varphi_L^*G'    \ar[r]^{} \ar[d] \cartesien&
A    \ar[r]    \ar[d]^{\varphi_L}& 0 \\
0 \ar[r] &
X^D    \ar[r]^{i'} &
G'    \ar[r]^{\pi'} &
A^*    \ar[r]    & 0
}
\]
Set theoretically $\varphi_L^*G'(S)= (G' \times_{A^*}A)(S)$ consists of pairs 
$(a,(E_{\varphi_L(a)},\vt))$ where $a\in A(S)$ and $(E_{\varphi_L(a)},\vt) \in 
G'(S),$ with $\vt : X \to E_{\varphi_L(a)}$
a trivialization of $v^*E_{\varphi_L(a)}$. The 
morphism $\ov{v} : X \to \varphi_L^*G'$ is then defined by
\[
 \ov{v}(y)=(v(y), (E_{\varphi_L(v(y))},\vt_{u(y)}))
\]
for any point $y\in X(S)$, where $\vt_{u(y)}$ is defined in 
equation~\eqref{eq:def_vtilde}.

Now we will construct a morphism $q : \varphi_L^*G'\to E_{\varphi_L(v(x))}$ 
that fits in the following commutative diagram:
\begin{equation}
\label{diag:def_q}
 \xymatrix{
0 \ar[r] &
X^D    \ar[r]^{}\ar[d]^{ev_x} &
\varphi_L^*G'    \ar[r]^{} \ar[d]^{q} &
A    \ar[r] \ar@{=}[d] & 0 \\
0 \ar[r] &
\Gm    \ar[r]^{} &
E_{\varphi_L(v(x))}    \ar[r]^{} &
A    \ar[r] & 0
}
\end{equation}
This will allow us to identify the pushdown ${ev_x}_*\varphi_L^*G'$ with 
$E_{\varphi_L(v(x))}$ and the extension ${ev_x}_*\ov{\varphi_L}^*M_L$ of 
$\ov{M}$ by $\Gm$ then corresponds to the pair $(E_{\varphi_L(v(x))}, q\circ 
\ov{v})$. The construction of $q$ is as follows. Let $(a,(E_{\varphi_L(a)},\vt))$ be an element of 
$\varphi_L^*G'(S)$, i.e. $a\in A(S)$ and $(E_{\varphi_L(a)},\vt) \in G'(S)$, with $\vt : X\to E_{\varphi_L(a)}$ an 
$A$-morphism. In particular we have a point $\vt(x) \in E_{\varphi_L(a)}(S)$ 
above $v(x)$, hence an isomorphism of line bundles $\beta : \varphi_L(a) \to 
\mu_{v(x)}^*\varphi_L(a)$. The latter isomorphism corresponds to a 
trivialization $\cO_A \simeq \mu_{v(x)+a}^*L\otimes \mu_{v(x)}^*L^{-1} \otimes 
\mu_{a}^*L^{-1} \otimes L$. Via the symmetry isomorphism, 
this in turn induces a trivialization of $\mu_{v(x)+a}^*L \otimes 
\mu_{a}^*L^{-1}\otimes 
\mu_{v(x)}^*L^{-1} \otimes L$, hence
an isomorphism of line bundles $\beta' : \varphi_L(v(x)) \to 
\mu_a^*\varphi_L(v(x))$. We define $q$ by
\[
 q(a,(E_{\varphi_L(a)},\vt)):=(a, \beta')
\]
with the above notation. In the diagram~\eqref{diag:def_q}, it is obvious that 
the right-hand square commutes. To prove that the left-hand square also 
commutes, we observe that 
both morphisms from $X^D$ to $E_{\varphi_L(v(x))}$ map an element $\alpha : 
X\to \Gm$ to the pair $(1,\alpha(x))$ where $1\in A(S)$ is the unit of $A$ and 
$\alpha(x)\in \Gm(S)$ is seen as an automorphism of the line bundle 
$\varphi_L(v(x))$. Now it follows from Lemma~\ref{lem:morphisme_du_milieu} that 
$q$ is automatically a group homomorphism.

We have proved that $u'(h^D(x))$ corresponds to the pair 
$(E_{\varphi_L(v(x))}, q\circ \ov{v})$. On the other hand, by definition of 
$\philt$, the extension $\philt(u(x))$ corresponds to the pair 
$(E_{\varphi_L(v(x))}, \vt_{u(x)})$. Hence to conclude the proof, it 
remains to prove that $q\circ \ov{v}=\vt_{u(x)}$. Let $y\in X(S)$ be a point of 
$X$ and let us prove that $q(\ov{v}(y))=\vt_{u(x)}(y)$. Unwinding the 
definitions of $q$, $\ov{v}$ and $\vt_{u(x)}$, we have to prove that the 
isomorphism $\varphi_{u(x),y} : \varphi_L(v(x)) \to 
\mu_{v(y)}^*\varphi_L(v(x))$ (see Lemma~\ref{lem:def_vtilde}) is equal to the 
isomorphism $\beta'$ induced by 
$\varphi_{u(y),x} : \varphi_L(v(y)) \to 
\mu_{v(x)}^*\varphi_L(v(y))$ via the symmetry isomorphism as explained in the 
previous paragraph (with $a=v(y)$). Since $\pi^*$ is faithful on the category 
of line bundles, it suffices to check the equality after applying $\pi^*$. In 
other words we have to prove that the descent datum $\delta$ on $\cL$ satisfies 
the following condition: $\mu_{u(x)}^*\delta_y\otimes \delta_y^{-1}$ should be 
equal to the isomorphism induced by $\mu_{u(y)}^*\delta_x\otimes \delta_x^{-1}$ 
through the symmetry isomorphism. But this is a consequence of the cocycle 
condition~\eqref{eq:CC} on the descent datum $\delta$ (use it both for $\delta_{x+y}$ and 
$\delta_{y+x}$). 
\end{proof}

This concludes the proof of Theorem~\ref{thm:existence_Phi}.
 We do not prove here that
$\Phi : \Pic(M)/\Pic(S) \to \Hom(M,M^*)$ (\ref{eq:PHI}) is a group homomorphism: 
this will follow from Corollary~\ref{cor:existence_phi_par_cube}, where we give a second construction of $\Phi$,
and from the comparison Theorem~\ref{thm:comparaison}.

We finish this Section giving another interesting construction of the morphism $\Phi : \Pic(M)/\Pic(S) \to \Hom(M,M^*)$ in the special case of Kummer 1-motives, that is 1-motives without abelian part. This construction, which is based on the second d\'evissage of the Picard group of $M$, involves only the group $\Lambda$ introduced in Definition~\ref{Lambda}.

 Let $M=[u: X \rightarrow T]$ be a Kummer 1-motive over a reduced 
scheme $S$.
In this case $M^*= [u^D: T^D \rightarrow X^D]$ and a morphism 
from $M$ to $M^*$ is a commutative diagram
\[
 \xymatrix{
X \ar[r]^g\ar[d]_u & T^{D}\ar[d]^{u^{D}} \\
T \ar[r]_{h} & X^{D} 
}
\]
By Definition~\ref{Lambda}, $\Lambda$ is a subgroup of 
$\Hom(M,M^*)$: an element $\lambda\in\Lambda$ defines the morphism $M \to 
M^*$ given by $\lambda: X \to T^D$ and $\lambda^D : T \to X^D $.

From Proposition~\ref{prop:devissage_kernel}, we know that the kernel $K$ of $\iota^* : 
\Pic(M)\to \Pic(T)$ fits in the exact sequence 
\[
 \Hom(T,\GG_m) \stackrel{\circ u}{\longrightarrow}  \Hom(X,\GG_m) 
\stackrel{\beta^*}{\longrightarrow} K \stackrel{\Theta}{\longrightarrow} \Lambda
 \stackrel{\Psi}{\longrightarrow} \Sigma.
\]
Then, locally on $S$, the morphism $\Phi : \Pic(M)\to \Hom(M,M^*)$ coincides 
with $\Theta$ in the following sense. Let $\cL$ be a line bundle on $M$. By 
Remark~\ref{rem:vanishing_Pic(T)} (2), 
since the tori underlying 1-motives are split locally for the \'etale topology, 
 there exists an \'etale and surjective morphism $S'\to S$ such that  
$(\iota^* \cL)_{|_{S'}}$ is trivial, which means that 
 $\cL_{|_{S'}} \in K$. 
Then $\Phi(\cL_{|_{S'}})$ is equal to $\Theta(\cL_{|_{S'}})$ via the inclusion 
$\Lambda\subset \Hom(M,M^*)$.

\begin{remark}\label{PhiNotSurjective} The homomorphism $\Phi : \Pic(M)/\Pic(S)\to \Hom(M,M^*)$ is far 
from being surjective. For example, let $M=[X\stackrel{u}{\rightarrow}T]$ with 
$X=\ZZ$, $T=\Gm$ and $u$ the trivial morphism. Then $\Hom(M,M^*)$ identifies 
with $\Hom(X,X)^2\simeq \ZZ^2$ and by Proposition ~\ref{prop:devissage_PicM}, 
the group
$\Pic(M)/\Pic(S)$ identifies with $\Hom(X,\Gm)\times \Lambda\simeq \Gm(S)\times 
\ZZ$. The morphism $\Phi : \Gm(S)\times \ZZ \to \ZZ^2$ is given by $(\gamma, 
n)\mapsto (n,n)$. 
\end{remark}

\section{Linear morphisms defined by cubical line bundles}
\label{section:linear_morphisms_def_by_cubical}

In this Section we first give the definition and basic properties of cubical 
structures on a line bundle over a commutative group 
stack $\cG$. Then we explain how a cubical line bundle on $\cG$, that is a line bundle on $\cG$ endowed with a cubical structure, defines an additive functor $\cG\to D(\cG)$ from $\cG$ to its dual.

Let $\cG$ be a commutative group stack over $S$, 
whose group law $(a,b) \mapsto ab$ will be denoted multiplicatively.
We denote by $\cG^3$ the commutative group stack
 $\cG\times_S \cG \times_S \cG$. Following \cite[Chp 
I, 2.4]{Moret_Bailly_Pinceaux}
we define a functor from the category of line bundles on $\cG$ to the category 
of line bundles on $\cG^3$
\begin{equation*}
\theta: \PIC(\cG) \longrightarrow \PIC(\cG^3)\\
\end{equation*}
with
\[
 \theta(\cL)=m_{123}^*\cL\otimes (m_{12}^*\cL)^{-1}\otimes (m_{13}^*\cL)^{-1}
\otimes (m_{23}^*\cL)^{-1}\otimes m_{1}^*\cL\otimes m_{2}^*\cL\otimes m_{3}^*\cL
\]
where for $I=\{i_1, \dots, i_l\}\subset \{1,2,3\}$, $m_{i_1\dots i_l}$ denotes 
the additive functor $\cG^3 \to \cG$ given by $(a_1, a_2, a_3) \mapsto 
a_{i_1}\dots a_{i_l}$. 
(Our $\theta(\cL)$ is denoted by $\theta_3(\cL)$ in \cite{Moret_Bailly_Pinceaux}.) 
 In terms of points the above definition becomes
\begin{equation}\label{eq:theta(L)}
 \theta(\cL)_{a_1, a_2, a_3}=\cL_{a_1a_2a_3}\otimes 
(\cL_{a_1a_2})^{-1}\otimes (\cL_{a_1a_3})^{-1}
\otimes (\cL_{a_2a_3})^{-1}\otimes \cL_{a_1}\otimes 
\cL_{a_2}\otimes \cL_{a_3}
\end{equation}
for any $(a_1,a_2,a_3) \in \cG^3$. As in \cite[Chp I, 
(2.4.2)]{Moret_Bailly_Pinceaux} the symmetric group $\Sgo_3$ of permutations 
acts on $\theta(\cL)$, that is for $(a_1,a_2,a_3)\in \cG^3$ and for $\sigma\in 
\Sgo_3$ there is a natural isomorphism
\begin{equation}
\label{isom:perm}
p^{\sigma}_{a_1,a_2,a_3} : \theta(\cL)_{a_1,a_2,a_3}
\stackrel{\sim}{\longrightarrow}
\theta(\cL)_{a_{\sigma(1)},a_{\sigma(2)},a_{\sigma(3)}}.
\end{equation}
Moreover, as in \cite[Chp I, (2.4.4)]{Moret_Bailly_Pinceaux}, $\theta(\cL)$ is 
endowed with cocycle isomorphisms:
for $a,b,c,d \in \cG$ one of these cocycle isomorphisms is
\begin{equation}
 \label{isom:coc}
coc_{a,b,c,d} : \theta(\cL)_{ab,c,d}\otimes \theta(\cL)_{a,b,d}
\stackrel{\sim}{\longrightarrow}
\theta(\cL)_{a,bc,d}\otimes\theta(\cL)_{b,c,d}\, ,
\end{equation}
the others are obtained from this one by permutation.

\begin{definition}
\label{def:cub}
 Let $\cL$ be a line bundle on $\cG$. A \emph{cubical structure on} $\cL$ is an 
isomorphism
$\tau : \cO_{\cG^3} \to \theta(\cL)$ of line bundles over $\cG^3$ that is 
compatible with the isomorphisms 
\eqref{isom:perm} and~\eqref{isom:coc}. In other words:
\begin{itemize}
 \item[(i)] For any $\sigma\in \Sgo_3$ and any $(a_1, a_2, a_3) \in \cG^3$, 
$\tau_{a_{\sigma(1)},a_{\sigma(2)},a_{\sigma(3)}}=
p^{\sigma}_{a_1,a_2,a_3}\circ \tau_{a_1,a_2,a_3}$.
\item[(ii)] For any $a,b,c,d\in \cG$,
$\tau_{a,bc,d}\otimes \tau_{b,c,d}
=coc_{a,b,c,d}\circ (\tau_{ab,c,d}\otimes \tau_{a,b,d})$.
\end{itemize}
A \emph{cubical line bundle on} $\cG$ is a pair $(\cL,\tau)$ where $\cL$ is a 
line bundle on $\cG$ and 
$\tau$ is a cubical structure on $\cL$. A \emph{morphism of cubical line bundles} 
$(\cL,\tau) \to (\cL',\tau')$ is a morphism $f : \cL\to \cL'$ of line bundles on 
$\cG$ such 
that $\tau'=\theta(f)\circ \tau$.
\end{definition}

We denote by $\CUB(\cG)$ the category of 
cubical line bundles on $\cG$, and by $\CUB^1(\cG)$ the group of isomorphism 
classes of 
cubical line bundles on $\cG$.

Let $\champcub(\cG)$ be the stack of cubical line bundles on $\cG$, i.e. 
for any $S$-scheme $U$, $\champcub(\cG)(U)$ is the category of 
cubical line bundles on $\cG \times_S U$. If $(\cL,\tau)$ and $(\cL',\tau')$ are 
two
cubical line bundles on $\cG$, then $\tau$ and $\tau'$ induce a canonical 
cubical 
structure on the line bundle $\cL\otimes \cL'$ and we denote by $(\cL,\tau)\otimes 
(\cL',\tau')$ the resulting cubical line bundle. The operation $\otimes$ endows
 $\champcub(G)$ with a structure of commutative group stack.

As in \cite[Chp I, 2.3]{Moret_Bailly_Pinceaux}
we also have a functor from the category of line bundles on $\cG$ to the 
category of line bundles on $\cG^2$
\begin{equation*}
\theta_2: \PIC(\cG) \longrightarrow  \PIC(\cG^2)
\end{equation*}
defined by
\[
 \theta_2(\cL)_{a,b}=\cL_{ab}\otimes \cL_a^{-1} \otimes \cL_b^{-1}
\]
for all $\cL\in \PIC(\cG)$ and all $(a,b) \in \cG^2$. This line bundle 
$\theta_2(\cL)$ furnishes
 a morphism of stacks  
\begin{eqnarray}
\nonumber \varphi_\cL : \cG  & \longrightarrow & \cHom_{S-stacks}(\cG,B\Gm)\\
\nonumber  a & \longmapsto & \big(\varphi_\cL(a):b \mapsto \varphi_\cL(a)(b)=\theta_2(\cL)_{a,b} \big).
\end{eqnarray}
 It is possible to recover $\theta(\cL)$ from $\theta_2(\cL)$ via the following two canonical isomorphisms
\[
\theta_2(\cL)_{ab,c}\otimes 
\theta_2(\cL)_{a,c}^{-1}\otimes \theta_2(\cL)_{b,c}^{-1}
 \simeq  \theta(\cL)_{a,b,c}
 \simeq \theta_2(\cL)_{a,bc}\otimes 
\theta_2(\cL)_{a,b}^{-1}\otimes \theta_2(\cL)_{a,c}^{-1}\, .
\]
Now let $\tau$ be a cubical structure on $\cL$. Through the above two isomorphisms, 
 $\tau$ induces two isomorphisms of line bundles (thought of as partial 
composition laws on $\theta_2(\cL)$):
\begin{align*}
 \tau_{a,b,c}^1 &: \theta_2(\cL)_{a,c}\otimes\theta_2(\cL)_{b,c}
\to \theta_2(\cL)_{ab,c} \\
\tau_{a,b,c}^2 &: \theta_2(\cL)_{a,b}\otimes\theta_2(\cL)_{a,c}
\to \theta_2(\cL)_{a,bc}\, .
\end{align*}
Generalizing \cite[Chp I, 2.5]{Moret_Bailly_Pinceaux} to line bundles on 
stacks, the conditions (i) 
and (ii) on $\tau$ imply that the two composition laws $\tau^1$ and $\tau^2$ are 
a structure  
of symmetric biextension of $(\cG,\cG)$ by $\Gm$ on the $\GG_m$-torsor 
$\theta_2(\cL)$ (see \cite[Definition 5.1]{Bertolin} for the notion of 
biextension of commutative group stacks). In particular, 
the isomorphism~$\tau^2$ provides for all points $a,b,c$ of $\cG$ a 
functorial isomorphism 
\[
 (\tau^2_{a,b,c})^{-1} : \varphi_{\cL}(a)(bc) \to 
\varphi_{\cL}(a)(b).\varphi_{\cL}(a)(c).
\]
The commutativity and associativity conditions that $\tau^2$ satisfies (see for 
instance the diagrams (1.1.3) and (1.1.5) p.2 in \cite{Breen_Fonctions_Theta}) 
imply that this isomorphism is compatible with the commutativity and 
associativity isomorphisms of $\cG$ and $B\Gm$. 
Hence $\varphi_\cL(a)$, equipped with this isomorphism, is an
additive functor from $\cG$ to $B\Gm$, that is $\varphi_\cL(a)$ is a point of 
$D(\cG)=\cHom (\cG, B\GG_m).$
This defines a morphism of 
stacks 
$$\varphi_\cL: \cG \longrightarrow D(\cG).$$
The isomorphism $(\tau^1)^{-1}$ defines 
a functorial isomorphism from $\varphi_\cL(ab)$ to 
$\varphi_\cL(a).\varphi_\cL(b)$ hence it endows $\varphi_{\cL}$ with the 
structure of an additive 
functor. The required compatibility conditions are given by the commutativity 
and associativity conditions on $\tau^1$ and by the compatibility 
of~$\tau^1$ and $\tau^2$ with each other (see \cite{Breen_Fonctions_Theta}, 
diagrams (1.1.4), (1.1.5) and (1.1.6)). From now on we denote by
$\varphi_{(\cL,\tau)}$ the resulting additive functor from $\cG$ to 
$D(\cG)$.

If $\alpha : (\cL,\tau) \to (\cL',\tau')$ is an isomorphism of cubical line 
bundles, the isomorphism $\theta_2(\alpha) : \theta_2(\cL)\to \theta_2(\cL')$ 
provides an isomorphism of functors from $\varphi_{(\cL,\tau)}$ to $\varphi_{(\cL',\tau')}$.
Since $\alpha$ is compatible with the cubical structures $\tau$ and $\tau'$, it follows that the 
latter isomorphism of functors is compatible with the additive structures of 
$\varphi_{(\cL,\tau)}$ and $\varphi_{(\cL',\tau')}$, in other words it is an 
isomorphism of additive functors, i.e. it is an isomorphism in 
$\cHom(\cG,D(\cG))$. This way the construction $(\cL,\tau)\mapsto 
\varphi_{(\cL,\tau)}$ is functorial and we get a morphism of stacks from 
${\mathcal Cub}(\cG)$ to ${\cHom}(\cG,D(\cG))$. Lastly, if $(\cL,\tau)$ and 
$(\cL',\tau')$ are two cubical line bundles, the canonical isomorphism 
$\theta_2(\cL\otimes \cL') \simeq \theta_2(\cL)\otimes \theta_2(\cL')$ 
(\cite[Chp I, 2.2.1]{Moret_Bailly_Pinceaux}) induces an 
isomorphism of functors from $\varphi_{(\cL,\tau)\otimes (\cL',\tau') }$ to 
$\varphi_{(\cL,\tau)}.\varphi_{(\cL',\tau')}$, which is compatible with the 
commutativity and associativity isomorphisms. Summing up, we have proved the following theorem.

\begin{theorem}
\label{thm:morphismes_induits_par_fi_cubistes}
 Let $\cG$ be a commutative group $S$-stack.
\begin{enumerate}
 \item Let $(\cL,\tau)$ be a cubical line bundle on $\cG$.
Then there is a natural additive functor $\varphi_{(\cL,\tau)} : \cG\to 
D(\cG)$, given by the formula
\[
\begin{array}{lrcl}
\varphi_{(\cL,\tau)} : & \cG & \longrightarrow & D(\cG) \\
    & a & \longmapsto & \big(b\mapsto \theta_2(\cL)_{a,b} = \cL_{ab}\otimes \cL_a^{-1} \otimes \cL_b^{-1}\big) \, .
\end{array}
\]
\item The above construction induces an additive functor
\[
 \begin{array}{lrcl}
\varphi : & {\mathcal Cub}(\cG) & \longrightarrow &
{\mathcal Hom}(\cG,D(\cG)) \\
    & (\cL,\tau) & \longmapsto & \varphi_{(\cL,\tau)}\, .
\end{array}
\]
\end{enumerate}
\end{theorem}

\begin{remark}
\label{rem:comparaison_cas_abelien}
 If $a$ is a point of $\cG$, the morphism $\varphi_{(\cL,\tau)}(a) : \cG\to 
B\Gm$ 
corresponds to the line bundle $(\mu_a^*\cL)\otimes (f^*a^*\cL)^{-1} \otimes 
\cL^{-1}$ on $\cG$, where $\mu_a : \cG\to \cG$ is the translation by $a$ and $f 
: \cG\to 
S$ is the structural morphism. In particular, if $\cG$ is an abelian $S$-scheme 
$A$, then 
$\varphi_{(\cL,\tau)}$ coincides with the 
classical morphism $\varphi_{\cL} : A \to A^*$ defined by $\varphi_{\cL}(a)= 
(\mu_a^*\cL)\otimes \cL^{-1}.$
By~\cite[VIII Prop 1.8]{Raynaud_notes_de_cours} $\varphi_{\cL}=0$ if and only 
if $\cL\in \Pic^0(A)$, hence $\varphi$ factorizes through the Néron-Severi 
group $NS(A)$ and induces $\ov{\varphi} : NS(A)\to \Hom(A,A^*)$. \end{remark}

\section{The theorem of the cube for 1-motives.}
\label{section:thm_cube}

If $\cG$ is a commutative group stack with neutral object $e$, we denote 
by $\TR(\cG)$ the category of line bundles on $\cG$ \emph{rigidified along 
$e$}, 
i.e. 
the category of pairs $(\cL,\xi)$ where $\cL$ is a line bundle on $\cG$ and $\xi 
: 
\cO_S \to e^*\cL$ is an isomorphism of line bundles. 

\begin{theorem}[Theorem of the cube for 1-motives]
\label{thm:cube}
 Let $S$ be a scheme. Let $[X \stackrel{u}{\to} G]$ be a complex of commutative $S$-group schemes. 
Assume that one of the following holds:
\begin{enumerate}
 \item $G$ is an abelian scheme.
\item $S$ is normal, $X\times_S X$ is reduced, $G$ is smooth 
with 
connected fibers, and the maximal fibers of $G$ are multiple extensions of 
abelian varieties, tori (not necessarily split) and groups $\GG_a$.
\end{enumerate}
Let $\cM= \st([X \stackrel{u}{\to} G])$ be the commutative group stack 
associated to the above complex via the equivalence of categories (\ref{st}). 
Then 
the forgetful functor
$$\CUB(\cM) \longrightarrow \TR(\cM)$$
is an equivalence of categories.
\end{theorem}

\begin{proof}
In the sequel, the group laws of $\cM$ and $G$ are denoted multiplicatively 
while the one of $X$ 
is denoted additively. We denote by $\iota : G\to \cM$ the canonical 
projection and by $1$ the unit section of~$G$. Then $\iota \circ 1 : S \to 
\cM$ is a neutral section of $\cM$ and will also be denoted by 1.

By (\ref{eq:theta(L)}) for 
any line bundle $\cL$ on $\cM$, there is a canonical isomorphism 
$\theta(\cL)_{1,1,1}\simeq \cL_1$, where $\cL_1$ is the line bundle $ 1^*\cL$ on 
$S$. Hence a cubical structure $\tau :
\cO_{\cM^3} \to \theta(\cL)$ on $\cL$ induces a natural rigidification of $\cL$
along the unit section that we still denote by
$\tau_{1,1,1} : \cO_S \to \cL_1$ (by a slight abuse of notation). The operation 
$(\cL,\tau) \mapsto (\cL, \tau_{1,1,1})$ defines a functor
$\CUB(\cM) \rightarrow \TR(\cM)$, which is the above-mentioned forgetful 
functor.
By \cite[Chp I, 2.6]{Moret_Bailly_Pinceaux} we already know that $G$ satisfies 
the theorem of the cube, i.e. 
that the forgetful functor $\CUB(G)\to \TR(G)$ is an equivalence of categories.

Let us prove that $\CUB(\cM)\to \TR(\cM)$ is fully faithful. Let $(\cL,\tau)$ 
and $(\cL', \tau')$ be two cubical line bundles on $\cM$ and let
 $f : \cL \to \cL'$ be a morphism in $\TR(\cM)$, i.e. a morphism which is 
compatible with the rigidifications $\tau_{1,1,1}$ and $\tau'_{1,1,1}$. We 
have to prove that $f$ is compatible with $\tau$ and $\tau'$, i.e. that 
$\tau'=\theta(f)\circ \tau$. Since the functor $\iota^*$ from the category of 
line bundles on $\cM$ to the category of line bundles on $G$ is faithful, this 
is 
equivalent to $\iota^*\tau'=\iota^*\theta(f)\circ \iota^*\tau$. But, up to 
canonical 
isomorphisms, $\iota^*\theta(f)$ identifies with $\theta(\iota^*f)$. Moreover, 
by 
assumption on $f$, $\tau'_{1,1,1}=f_1\circ \tau_{1,1,1}$, hence 
$(\iota^*\tau')_{1,1,1}=(\iota^*f)_1\circ(\iota^*\tau)_{1,1,1}$. This means 
that 
$\iota^*f 
: \iota^*\cL \to \iota^*\cL'$ is compatible with the rigidifications induced by 
the 
cubical structures $\iota^*\tau$ and $\iota^*\tau'$ on $\iota^*\cL$ and 
$\iota^*\cL'$. 
By the theorem of the cube for $G$, this implies the desired equality 
$\iota^*\tau'=\theta(\iota^*f)\circ \iota^*\tau$.

Now let us prove that $\CUB(\cM)\to \TR(\cM)$ is essentially surjective. As 
observed at the end of Section 1, a line bundle $\cL$ on $\cM$ is a pair
$(L,\delta)$ 
where $L=\iota^*\cL$ is a line bundle on~$G$ and $\delta : p_2^*L\to \mu^*L$ is 
a descent datum for $L$. Let $\xi : \cO_S\to \cL_1$ be a rigidification of 
$\cL$ 
along the unit section of $\cM$. Via the canonical 
isomorphism $\cL_1\simeq L_1$, $\xi$ is also a rigidification of $L$ along the 
unit section of $G$. By the theorem of the cube for $G$, there is a cubical 
structure $\tau : \cO_{G^3} 
\to \theta(L)$ that induces~$\xi$, i.e. such that $\tau_{1,1,1}=\xi$.
We want to construct a cubical structure $\ov{\tau} : \cO_{\cM^3} \to 
\theta(\cL)$ that induces $\xi$. The group 
stack $\cM^3$ is canonically 
isomorphic to the quotient stack $[G^3/X^3]$ with the action of $X^3$ on $G^3$ 
by translations via $u^3 : X^3 \to G^3$.  As for $\cM$, we identify the 
category 
of line bundles on $\cM^3$ with the category of line bundles on $G^3$ equipped 
with a descent datum. The line bundle $\cO_{\cM^3}$ corresponds to $\cO_{G^3}$ 
equipped with the canonical isomorphism $p_2^*\cO_{G^3} \to \mu^*\cO_{G^3}$ 
(where $p_2, \mu : X^3\times_S G^3 \to G^3$ respectively denote the second 
projection and the action by translation). The line bundle 
$\theta(\cL)$ on $\cM^3$ corresponds to the line bundle~$\theta(L)$ on~$G^3$ 
equipped with the descent datum $p_2^*\theta(L)\simeq \theta(p_2^*L) 
\stackrel{\theta(\delta)}{\to} \theta(\mu^*L) \simeq \mu^*\theta(L)$, that by 
a slight abuse we denote by $\theta(\delta)$. In terms of points, 
$\theta(\delta)$ can be described as follows: for any points $x=(x_1, x_2, 
x_3)$ of $X^3$ and $a=(a_1, a_2, a_3)$ of $G^3$,
\begin{equation}
\label{eq:description_theta_delta}
  \theta(\delta)_{x,a} : \theta(L)_a \to \theta(L)_{u^3(x)a}
\end{equation}
is equal to $\delta_{x_1+x_2+x_3,a_1a_2a_3}
\otimes \delta_{x_1+x_2,a_1a_2}^{-1}
\otimes \delta_{x_1+x_3,a_1a_3}^{-1}
\otimes \delta_{x_2+x_3,a_2a_3}^{-1}
\otimes \delta_{x_1,a_1}
\otimes \delta_{x_2,a_2}
\otimes \delta_{x_3,a_3}.$

We claim that the following diagram of line bundles on $X^3\times_S G^3$ 
commutes
\begin{equation}
\label{diag:tau_compat_phi}
 \xymatrix{p_2^*\cO_{G^3}
\ar[r]^{can.} \ar[d]_{p_2^*\tau} &\mu^*\cO_{G^3}
\ar[d]^{\mu^*\tau} \\
p_2^* \theta(L)\ar[r]_{\theta(\delta)} & \mu^*\theta(L)\, .
}
\end{equation}
The proof of this claim will be the main part of the proof. It is equivalent 
to saying that for any points $x$ of $X^3$ and $a$ of $G^3$, we have 
$\theta(\delta)_{x,a}\circ \tau_a=\tau_{u^3(x)a}$.
For any $S$-scheme $U$, we identify $\Aut(\cO_U)$ with $\Gm(U)$ and this allows 
us to define a morphism of $S$-schemes
\[
 \begin{array}{lrcl}
\lambda : & X^3\times_SG^3 & \longrightarrow &
\Gm \\
    & (x,a) & \longmapsto &\tau_{u^3(x)a}^{-1}\circ\theta(\delta)_{x,a}\circ 
\tau_a\, .
\end{array}
\]
Now to prove the claim we have to prove that $\lambda$ is constant equal to 1.

By \eqref{eq:CC}, the following diagram commutes
$$\xymatrix{
\theta(L)_a \ar[rr]^{\theta(\delta)_{x+x',a}}
\ar[rd]_{\theta(\delta)_{x',a}}
&& \theta(L)_{u^3(x)u^3(x')a} \\
& \theta(L)_{u^3(x')a}\ar[ru]_{\theta(\delta)_{x,u^3(x')a}}
&
}$$
It follows that for any $x,x'\in X^3$ and any $a\in G^3$ we have the equation
\begin{align}
\label{eq:CCter}
 \lambda(x+x',a)=\lambda(x,u^3(x')a).\lambda(x',a)\, .
\end{align}
For any $x\in X^3, a\in G^3$ and any 
permutation $\sigma\in \Sgo_3$, by the condition~(i) of 
Definition~\ref{def:cub}, 
the left and right triangles in the following diagram commute (where for 
$a=(a_1,a_2,a_3)$ we write 
$a^{\sigma}=(a_{\sigma(1)},a_{\sigma(2)},a_{\sigma(3)})$)
$$\xymatrix@C=3,5pc@R=1pc{
&\theta(L)_a
\ar[r]^-{\theta(\delta)_{x,a}} \ar[dd]_{p_a^{\sigma}}
& \theta(L)_{u^3(x)a}
\ar[dd]^{p_{u^3(x)a}^{\sigma}} \\
\cO_U \ar[ru]^{\tau_a}
\ar[rd]_{\tau_{a^{\sigma}}} &&&
\cO_U \ar[lu]_{\tau_{u^3(x)a}}
\ar[ld]^-{\tau_{u^3(x^{\sigma})a^{\sigma}}} \\
&\theta(L)_{a^{\sigma}}
\ar[r]_-{\theta(\delta)_{x^{\sigma},a^{\sigma}}}
& \theta(L)_{u^3(x^{\sigma})a^{\sigma}}
}$$
The central square also commutes by construction of the canonical isomorphism 
$p_a^{\sigma}$ and of~$\theta(\delta)$. Hence
\begin{align}
 \label{eq:perm}
\lambda(x^{\sigma},a^{\sigma})=\lambda(x,a).
\end{align}

Now let us choose $x\in X^3$ and $a\in G^3$ such that $x_3=0$ and $a_3=1$. From 
the above description~\eqref{eq:description_theta_delta} of $\theta(\delta)$ we 
see that, via the canonical 
isomorphisms $\theta(L)_a\simeq \theta(L)_{1,1,1}$ and 
$\theta(L)_{u^3(x)a}\simeq \theta(L)_{1,1,1}$, the isomorphism 
$\theta(\delta)_{x,a}$ is just the identity of $\theta(L)_{1,1,1}$. Moreover, 
as in \cite[Chp I, 2.5.3]{Moret_Bailly_Pinceaux}, from condition 
(ii) of Definition \ref{def:cub} it follows that 
$\tau_a=\tau_{u^3(x)a}=\tau_{1,1,1}$. Using 
\eqref{eq:perm}, we get 
\begin{equation}
\label{eq:triv}
\lambda(x,a)=1
\end{equation}
as soon as there is an index $i$ 
such that $x_i=0$ and $a_i=1$. In particular, if $x_i=0$ for some $i$, we 
have $\lambda(x,1)=1$. Hence Lemma~\ref{lemma:invertible_sections_on_a_sav},
applied to the $S$-group scheme $G^3$, implies that $\lambda$ is a group 
homomorphism 
in the variable $a$, i.e. for any $x\in X^3$ such that some $x_i$ is zero, and 
for any $a,a'\in G^3$ we have
\begin{align}
\label{eq:gpe}
 \lambda(x,aa')=\lambda(x,a).\lambda(x,a')
\end{align}
[Actually Rosenlicht only applies when the base scheme $S$ is reduced. But we
apply it for the ``universal'' point $(\id_{X\times_S X},0)\in X^3(U)$ where 
the base scheme $U=X\times_S X$ is reduced, and the general case follows.]
In particular for $x=(x_1,0,0)\in X^3$ and for any $a=(a_1,a_2,a_3)\in G^3$, 
using \eqref{eq:gpe} and \eqref{eq:triv} we get
\[
 \lambda(x,a)=\lambda(x,(a_1,a_2,1))\lambda(x,(1,1,a_3))=1
\]
By \eqref{eq:perm} this proves that $\lambda(x,a)=1$ as soon as two of the 
$x_i$'s are zero and finally using \eqref{eq:CCter} this proves that $\lambda$ 
is constant equal to 1. This finishes the proof of the claim.

Now, the commutativity of (\ref{diag:tau_compat_phi}) means that
$\tau$ is an isomorphism in the category of line bundles on $G^3$ equipped with
descent data. Hence it corresponds to an isomorphism $\ov{\tau} :
\cO_{\cM^3} \to \theta(\cL)$. Moreover, the condition
 (i) (resp. (ii)) of Definition \ref{def:cub} can be expressed by the
commutativity of
some diagrams of line bundles over $\cM^3$ (resp. $\cM^4$). Since the
functor~$\iota^*$ is faithful, the fact that $\tau$ satisfies
the conditions (i) and (ii) of Definition \ref{def:cub} implies that $\ov{\tau}$
itself satisfies
these two conditions. Hence $\ov{\tau}$ is a cubical structure on $\cL$. From 
$\tau_{1,1,1}=\xi$ it follows that $\ov{\tau}_{1,1,1}=\xi$ and this concludes 
the proof of the theorem.
\end{proof}

\begin{corollary}
\label{cor:existence_phi_par_cube}
 With the notation and assumptions of Theorem~\ref{thm:cube}, there is a functorial 
group homomorphism $\Phi' : \Pic(\cM)/\Pic(S) \to \Hom(\cM,D(\cM))$. 
\end{corollary}
\begin{proof}\ 
Since $\Pic(\cM)/\Pic(S)$ is isomorphic to the group of isomorphism classes of 
rigidified line bundles on $\cM$, this 
is an immediate consequence of Theorems~\ref{thm:morphismes_induits_par_fi_cubistes} 
and~\ref{thm:cube}.
\end{proof}

\begin{theorem}
 \label{thm:comparaison}
 Let $M$ be a 1-motive defined over a scheme $S$. Assume that the base scheme $S$ is normal. The morphism 
$\Phi'$ defined above coincides with the morphism 
$\Phi : \Pic(M) / \Pic(S) \to \Hom(M,M^*)$ constructed in 
Section~\ref{section:direct_construction}.
\end{theorem}
\begin{proof}
Let $(\cL,\delta)$ be line bundle on $M$. We want to prove that 
$\Phi(\cL,\delta)=\Phi'(\cL,\delta)$. The question is local on $S$ hence 
as in section~\ref{section:direct_construction} we may assume that the line 
bundle $\cL$ on $G$ is induced by a line 
bundle $L$ on $A$, i.e. $\cL=\pi^*L$. To prove the theorem it suffices to prove 
that the morphisms $A\to A^*$, $X\to T^D$ and $T\to X^D$ induced by 
$\Phi'(\cL,\delta)$ 
are respectively equal to the $\varphi_L$, $h^D$ and $h$ of
section~\ref{section:direct_construction}.
 
The Cartier dual of $G$ as a 1-motive is 
$G^*=[T^D\stackrel{v'}{\rightarrow}A^*]$ and $\Hom(G,G^*)=\Hom(A,A^*)$. By 
functoriality of $\Phi'$, the morphisms $\iota : G\to M$ and $\pi : G\to A$ 
induce a commutative diagram:
\[
\xymatrix{
\Pic(M) \ar[r]^{\iota^*} \ar[d]_{\Phi'} &
\Pic(G) \ar[d]^{\Phi_G'} &
\Pic(A) \ar[l]_{\pi^*} \ar[d]^{\Phi_A'} \\
\Hom(M,M^*) \ar[r] &
\Hom(G,G^*) &
\Hom(A,A^*) \ar[l]_{\sim} 
}
\]
The morphism $A\to A^*$ induced by $\Phi'(\cL,\delta)$ is the image of 
$\Phi'(\cL,\delta)$ under the bottom horizontal map of this diagram. Hence it 
is equal to $\Phi'_A(L)$, which is equal to $\varphi_L$ by 
Remark~\ref{rem:comparaison_cas_abelien}.

Now let us prove that the morphism $\xi : T\to X^D$ induced by 
$\Phi'(\cL,\delta)$ is 
equal to $h$. To this end we consider the action of $\Phi'(\cL,\delta)$ on the 
objects of $\st(M)$. Let $t\in T(S)$ be a point of $T$. Its image 
$i(t)\in G(S)$ induces an object of the stack $\st(M)$ still denoted by $i(t)$, 
and by definition $\Phi'(\cL,\delta)(i(t))$ is the morphism from $\st(M)$ to 
$B\Gm$ that 
maps an object $b$ to $\theta_2(\cL)_{i(t), b}$. To get the induced morphism 
from $X$ to $\Gm$ it suffices to consider the action of 
$\Phi'(\cL,\delta)(i(t))$ 
on the 
arrows of the stack $\st(M)$. If $b_1, b_2 \in G(S)$ 
and if $x\in X(S)$ is an arrow from $b_1$ to $b_2$ in $\st(M)$ (i.e. 
$u(x)=b_2-b_1$) then  
$\Phi'(\cL,\delta)(i(t))$ maps this arrow to the induced isomorphism from
$\theta_2(\cL)_{i(t),b_1}$ to $\theta_2(\cL)_{i(t),b_2}$. The induced element 
$\xi(t)(x)\in \Gm(S)$ does not depend on the choice of the source $b_1$ hence 
we may chose 
$b_1=1$ and $\xi(t)(x)$ is the point of $\Gm(S)$ induced by the isomorphism 
$\theta_2(\cL)_{i(t),1}\to \theta_2(\cL)_{i(t),u(x)}$ induced by $\delta$. The 
latter is~$\delta_{x,i(t)}\otimes \delta_{0,i(t)}^{-1} \otimes 
\delta_{x,1}^{-1}$. But, by the cocycle condition~\eqref{eq:CC},
$\delta_{0,i(t)}$ is the identity, hence this corresponds to the description of 
$h$ given in Remark~\ref{rem:description_h}.

To prove that $\Phi'(\cL,\delta)$ induces $h^D$ from $X$ to $T^D$ we have to 
consider its action on the arrows of $\st(M)$. The argument is very similar to 
the above one and left to the reader.
\end{proof}

\begin{remark}
\label{rem:comments_normalness}
The hypothesis of normalness on $S$ is essential in order to identify the 
categories of cubical line bundles with the categories of line bundles 
rigidified along the unit section, even on a torus. See \cite[Chp I, Example 
2.6.1]{Moret_Bailly_Pinceaux} for a counter-example.
Hence if the base scheme $S$ is not normal, we only have the functorial 
homomorphism $\CUB^1(M)\to \Hom(M,M^*)$ given 
by Theorem~\ref{thm:morphismes_induits_par_fi_cubistes}. The morphism $\CUB^1(M) \to 
\Pic (M) /\Pic(S)$ induced by the forgetful functor $\CUB(M)\to \TR(M)$ is 
neither injective nor surjective in general. If $S$ is reduced, we can prove 
that the forgetful functor is fully faithful, hence $\CUB^1(M) \to 
\Pic (M) /\Pic(S)$ is injective. This inclusion is an isomorphism if the base 
scheme $S$ is normal.
\end{remark}



\end{document}